\documentclass{amsart}

\usepackage[utf8]{inputenc}
\usepackage{amssymb}
\usepackage{array}
\usepackage{tikz-cd}
\usepackage{a4wide}
\usepackage{amsmath}
\usepackage{graphicx}
\usepackage{color}
\usepackage{hyperref}
\usepackage{mathtools}
\usepackage{amsrefs}

\definecolor{dubiousEnglish}{rgb}{1,0,.2}
\definecolor{darkgreen}{rgb}{0,0.40,0}

\DeclareMathOperator{\ImaginaryPart}{Im}

\DeclareMathOperator{\Int}{Int}

\DeclareMathOperator{\RealPart}{Re}

\DeclareMathOperator{\Diff}{Diff}

\DeclareMathOperator{\e}{\operatorname{e}}

\newcommand{\restricted}[2]{{\left.{#1}\right|_{#2}}}
\newcommand{\lcan}{{\lambda_{\mathrm{can}}}}

\newcommand{\p}{\partial}
\newcommand{\lie}[1]{{\mathcal{L}_{#1}}}
\newcommand{\abs}[1]{{\left\lvert #1\right\rvert}}
\newcommand{\norm}[1]{{\lVert #1\rVert}}

\renewcommand{\epsilon}{\varepsilon}
\makeatletter
\usepackage{xspace}
\usepackage{enumerate}
\usepackage{mathtools}

\makeatother

\newcommand{\CC}{{\mathbb C}}

\newcommand{\DD}{{\mathbb D}}

\newcommand{\fF}{{\mathcal F}}

\newcommand{\bfp}{{\mathbf p}}
\newcommand{\bfq}{{\mathbf q}}

\newcommand{\RR}{{\mathbb R}}
\renewcommand{\SS}{{\mathbb S}}
\newcommand{\TT}{{\mathbb T}}

\newcommand{\bfu}{{\mathbf u}}

\newcommand{\x}{{\mathbf x}}
\newcommand{\y}{{\mathbf y}}
\newcommand{\z}{{\mathbf z}}
\newcommand{\ZZ}{{\mathbb Z}}

\newcommand{\defin}[1]{\textbf{#1}}

\newcommand{\vol}{\mathrm{vol}}

\theoremstyle{plain}

\newtheorem{theorem}{Theorem}[section]
\newtheorem{lemma}[theorem]{Lemma}

\newtheorem{corollary}[theorem]{Corollary}

\newtheorem{openquestion}[theorem]{Question}
\newtheorem{proposition}[theorem]{Proposition}

\theoremstyle{remark}
\newtheorem{remark}[theorem]{Remark}
\newtheorem{example}[theorem]{Example}

\theoremstyle{definition}
\newtheorem{definition}{Definition}

\newcounter{maintheorem}

\newtheorem{main_theorem}[maintheorem]{Theorem}
\newtheorem*{theorem_inverse_monodromy}{Theorem~\ref{thm: inverse
monodromy same Bourgeois structure}}
\newtheorem*{theorem_weak_filling}{Theorem~\ref{thm: fillings for V
that imply fillings for VxT2}.(a)}
\newtheorem*{theorem_subcrit_filling}{Theorem~\ref{thm: fillings for V
that imply fillings for VxT2}.(b)}
\newtheorem*{theorem_obstruction_subcrit}{Theorem~\ref{thm: main
obstruction result}}

\numberwithin{equation}{section}

\usepackage{framed}
\usepackage[normalem]{ulem}
\author[S.\ Lisi]{Samuel Lisi}
 
\address[S.\ Lisi]{
  University of Mississippi \\
  Department of Mathematics \\
  P.O. Box 1848 \\
  University, MS 38677-1848\\
  USA}

\email{stlisi@olemiss.edu}

\author[A.\ Marinković]{Aleksandra Marinković}

\address[A.\ Marinković]{
  Matematicki fakultet\\
  Studentski trg 16\\
  11 000 Belgrade\\
  SERBIA}

\email{aleks@matf.bg.ac.rs}

\author[K.\ Niederkrüger]{Klaus Niederkrüger}

\address[K.\ Niederkrüger]{
  Institut Camille Jordan \\
  Université Claude Bernard Lyon 1\\
  43 boulevard du 11 novembre 1918\\
  F-69622 Villeurbanne Cedex\\
  FRANCE}

\email{niederkruger@math.univ-lyon1.fr}

\title{On properties of Bourgeois contact structures}

\begin{document}

\begin{abstract}
  The Bourgeois construction associates to every contact open book on
  a manifold~$V$ a contact structure on $V\times \TT^2$.
  We study in this article some of the properties of $V$ that are
  inherited by $V\times \TT^2$ and some that are not.
  Giroux has provided recently a suitable framework to work with
  contact open books.
  In the appendix of this article, we quickly review this formalism,
  and we work out a few classical examples of contact open books to
  illustrate how to use this new language.
\end{abstract}

\maketitle

\tableofcontents

\section{Introduction}

In his thesis, Bourgeois used a construction based on work by Lutz
\cite{Lutz_invariantes} that associates to every contact open book on
a contact manifold~$(V,\xi)$ a contact structure on $V\times \TT^2$
that is invariant under the natural $\TT^2$-action and that restricts
on every fiber $V\times\{*\}$ to $\xi$, see \cite{BourgeoisTori}.
Even though all contact structures obtained on $V\times \TT^2$ for
a given $(V, \xi)$ are homotopic as \emph{almost} contact structures
independently of the open book used, Bourgeois proved via contact
homology that the resulting contact structures on $V\times \TT^2$
often do depend on the open book chosen and not only on $\xi$ itself.
This construction is probably the most interesting explicit method
known so far to produce higher dimensional closed contact manifolds
based on lower dimensional ones.
For this reason we consider it an important question to understand
which properties of $(V,\xi)$ are passed on to the associated contact
structure on $V\times \TT^2$.
For instance, Presas constructed the first examples of higher
dimensional overtwisted contact structures
\cite{PresasExamplesPlastikstufes} by gluing together two Bourgeois
structures associated to overtwisted $3$-manifolds.
This raised the question of whether the Bourgeois structure associated
to an overtwisted structure is overtwisted or not. We will show here that 
this is not always the case. 
The list of properties we will be studying are mostly related to the
fillability and tightness of the Bourgeois structures.
Note also the recent article \cite{GironellaBourgeoisStructs} by
Gironella that studies questions about Bourgeois structures related to
ours. 
We discuss the relation of our work to his in Section~\ref{sec:
  bourgeois structures introduction}.
Recall that a general contact structure is either overtwisted or tight
\cite{BormanEliashbergMurphyExistence}, furthermore it is known that
overtwisted manifolds are not even weakly fillable
\cites{NiederkrugerPlastikstufe, WeakFillabilityHigherDimension} (to
drop the semipositivity condition use \cite{Pardon}).
The different types of fillability can be combined to give the
following hierarchy:
\begin{quote}
  subcritically Weinstein fillable $\Rightarrow$ Weinstein fillable
  $\Rightarrow$ exact fillable $\Rightarrow$ strongly fillable
  $\Rightarrow$ weakly fillable $\Rightarrow$ tight.
\end{quote}
For Bourgeois contact structures, we know from
\cite{WeakFillabilityHigherDimension} (and work related to it):

\begin{main_theorem}\label{thm: fillings for V that imply fillings for
    VxT2}
  Let $(V,\xi)$ be a closed contact manifold.
  \begin{itemize}
  \item [(a)] If $(V,\xi)$ is weakly filled by $(W,\omega)$, then
    independently of the open book decomposition used in the
    construction, the associated Bourgeois contact structure on
    $V\times \TT^2$ is isotopic to a contact structure that can be
    weakly filled by $(W\times \TT^2, \omega\oplus \vol_{\TT^2})$.
  \item [(b)] If $(V,\xi)$ admits a Weinstein filling that is a
    $k$-fold stabilization, and if $(K,\vartheta)$ is the canonical
    open book associated to such a subcritical filling, then the
    corresponding Bourgeois structure on $V\times \TT^2$ will be
    $(k-1)$-subcritically Weinstein fillable.
  \end{itemize}
\end{main_theorem}

We draw the reader's attention to two different meanings of
``stabilization'' in this paper.
In the context of Weinstein domains, this refers to taking a product
with $\CC$ (or $\CC^k$), see Section~\ref{sec: explicit fillings} for
details.
In the context of an (abstract) open book, however, it refers to a
modification of the open book by attaching a handle to the page and
also changing the monodromy by a suitable Dehn twist.
See, for instance, \cite{KoertLectOpenBooks}*{Section 4.3}.
In Section~\ref{sec: bourgeois structures introduction} we explain the
Bourgeois {construction.
The proof of Theorem~\ref{thm: fillings for V that imply fillings for
  VxT2} is in Section~\ref{sec: explicit fillings}.}
As already mentioned, the Bourgeois structures do not only depend on
the chosen contact manifold~$(V,\xi)$ but also on the open book used
in the construction \cite{Bourgeois_thesis}.
On the other hand, two abstract open books with the same page but with
mutually inverse monodromies, $\Psi$ and $\Psi^{-1}$, lead to two
contact manifolds that are smoothly (orientation reversing)
diffeomorphic to each other but that, in general, have very different
contact properties.
For example, from Giroux \cite{Giroux_ICM}, a contact manifold
  is Stein/Weinstein fillable if and only if it admits an open book
whose monodromy~$\Psi$ can be expressed as a product of positive Dehn
twists.
By contrast, changing the monodromy of an abstract open book
to $\Psi^{-1}$ often yields an overtwisted contact structure.
Nonetheless we obtain the following unexpected result in
Section~\ref{section: inverse monodromy}.

\begin{main_theorem}\label{thm: inverse monodromy same Bourgeois
    structure}
  Let $(V,\xi_+)$ and $(V,\xi_-)$ be closed contact manifolds
  supported by abstract Liouville open books that have the same page
  but inverse monodromy.
  Then the two corresponding Bourgeois structures on $V\times \TT^2$
  are contactomorphic.
\end{main_theorem}

This statement shows that the Bourgeois construction is not injective,
and combining this result with Theorem~\ref{thm: fillings for V that
  imply fillings for VxT2}, we also obtain the following corollary.

\begin{corollary}
  There exist examples in every dimension of $(V, \xi)$ closed
  overtwisted contact manifolds for which at least one of the
  corresponding Bourgeois structures on $V \times \TT^2$ is tight.
\end{corollary}

In fact, no example of an overtwisted Bourgeois structure is known to
us.
Note also that Gironella has recently shown that every contact
$3$-manifold with non-trivial fundamental group admits an open book
whose Bourgeois structure is (hyper)-tight
\cite{GironellaBourgeoisStructs}.
This leads to the following questions.

\begin{openquestion}
  \begin{itemize}
  \item [(a)] Can a Bourgeois contact structure ever be overtwisted?
  \item [(b)] Are there Bourgeois contact structures that are not
    weakly fillable?
  \end{itemize}
\end{openquestion}

In both cases, it is an immediate consequence of Theorem~\ref{thm:
  inverse monodromy same Bourgeois structure} that candidates can only
be constructed from open books where both the monodromy and the
inverse monodromy lead to overtwisted or not weakly fillable contact
structures, respectively.

\vspace{0.5cm}

In Section~\ref{sec: subcritical fillings}, we show that most
Bourgeois structures are not subcritically Weinstein fillable.
Subcritically fillable contact manifolds are extremely rare --- in
dimension~$3$ the only examples are the standard sphere and connected
sums of copies of $\SS^1\times \SS^2$ with the tight contact
structure.
In high dimensions, the comprehensive study of the topological
characterization of Stein fillable manifolds was conducted by Bowden,
Crowley, and Stipsicz \cite{BowdenCrowleyStipsicz1}.
Let $(V,\Xi_V, \omega_\Xi)$ be an almost contact manifold, that is,
$V$ is an oriented manifold with a hyperplane field~$\Xi_V$ and
$\omega_\Xi$ is a symplectic structure on $\Xi_V$.
An \defin{almost Stein filling~$(W,J)$} of $(V,\Xi_V, \omega_\Xi)$ is
an almost complex manifold such that
\begin{itemize}
\item $V$ is the \emph{oriented} boundary of $W$;
\item $J$ restricts to $\Xi_V$ and $\restricted{J}{\Xi_V}$ is tamed by
  $\omega_\Xi$;
\item $W$ admits a handle decomposition with all handles of dimension
  no more than $\frac{1}{2}\,\dim W$.
\end{itemize}
In particular, \cite{BowdenCrowleyStipsicz1}*{Proposition~7.1}
specializes in our situation to the following.

\begin{theorem}
  Let $(V,\Xi_V, \omega_\Xi)$ be an almost contact structure, and let
  $d\vol$ be a volume form on $\TT^2$.
  If $(V\times \TT^2, \Xi_V \oplus T\TT^2, \omega_\Xi \oplus d\vol)$
  admits an almost Stein filling, it follows that
  $(V,\Xi_V, \omega_\Xi)$ also admits one.
  \qedhere
\end{theorem}

Conversely, if $(V,\Xi_V, \omega_\Xi)$ admits a
subcritical almost Stein filling, then
$(V\times \TT^2, \Xi_V \oplus T\TT^2, \omega_\Xi \oplus d\vol)$ admits
an almost Stein filling.
Compare this also to part~(b) of Theorem~\ref{thm: fillings for V that
  imply fillings for VxT2}.
If $(V,\xi_V)$ is only Stein fillable, however, the situation is
significantly more involved.
For example, $\TT^3$ with the standard contact structure has the Stein
filling $T^*\TT^2$.
No contact structure on $\TT^3\times \TT^2 = \TT^5$ can ever be Stein
fillable by part~(2) of
\cite{BowdenCrowleyStipsicz1}*{Proposition~6.2}.
We give below a few examples of Bourgeois structures {that 
    admit subcritical \textit{almost} Stein fillings 
    but no \textit{genuine} subcritical fillings.
  This is based on obstructions to Weinstein fillability that} can easily be
deduced from Gromov's '85 article \cite{Gromov_HolCurves}.

\begin{main_theorem}\label{thm: main obstruction result}
  A closed contact manifold containing a weakly exact
  pre-Lagrangian~$P$ is not subcritically Weinstein fillable.
  If the dimension of the contact manifold is at least $5$ and if $P$
  is displaceable then the contact manifold is not even Weinstein fillable.
\end{main_theorem}

The major draw-back of this easy theorem is that pre-Lagrangians can
only be weakly exact in manifolds with a sufficiently large
fundamental group (see Lemma~\ref{lemma: exact pre-Lagrangians and
  positive loops}), thus excluding many interesting cases.
There is little doubt that this limitation could be somewhat relaxed
by using some type of Floer theory, but we refrain from doing so to
keep this note simple.
Note also that with \cites{OanceaViterbo, BarthGeigesZehmisch} one can
formulate obstructions to subcritical fillability that depend more on
the global topology of the contact manifold.
On the other hand, Theorem~\ref{thm: main obstruction result} leads to
the following observation regarding Bourgeois structures:

\begin{corollary}\label{cor: page contains Legendrian then Bourgeois
    not subcritical}
  Let $(V,\xi)$ be a closed contact manifold and let $(K,\vartheta)$
  be a compatible open book that contains a closed Legendrian in one
  of its pages.
  It then follows that the corresponding Bourgeois structure on
  $V\times \TT^2$ is not subcritically Weinstein fillable.
  This applies in particular to any open book that has been
  stabilized. 
\end{corollary}

\begin{example}\label{example: bourgeois from non-trivial open book not subcritical}
  Consider the contact open book decomposition of the standard contact
  sphere $(\SS^{2n-1}, \xi_0)$ whose page is a ball and whose
  monodromy is trivial.
  If $2n-1 \ne 1$, the corresponding Bourgeois structure on
  $\SS^{2n-1}\times \TT^2$ is subcritically Weinstein fillable.
  If instead we take for example an open book with page the cotangent
  bundle~$T^*\SS^{n-1}$ and with monodromy a positive Dehn twist
  (these examples are classical but they are also explained in depth
  in the appendix), then the Bourgeois structure will be homotopic to
  the first one as almost contact structures, but it cannot be
  contactomorphic\footnote{Note that these two examples were
    explicitly excluded in the contact homology computations in
    Bourgeois' thesis.} to it, since it is not subcritically Weinstein
  fillable.
\end{example}

This way, we see that the fillability of a Bourgeois structure on
$V\times \TT^2$ depends not only on the contact manifold~$(V,\xi)$ but
also on the open book used in the construction.
These results should be severely improved, and in particular it would
be nice to find an answer to the following question:

\begin{openquestion}
  Are the Bourgeois structures on $\SS^{2n-1}\times \TT^2$ obtained
  from the standard contact sphere~$(\SS^{2n-1}, \xi_0)$ and the open
  book decomposition whose page is $T^*\SS^{n-1}$
  (Example~\ref{example: contact open book decomposition for
    sphere}.(b) and \ref{example: from open book to abstract open
    book}.(b)) strongly fillable?
\end{openquestion}

We know from Theorem~\ref{thm: fillings for V that imply fillings for
  VxT2} that these are weakly fillable.
If they were not strongly fillable they would provide the first
examples of weakly but not strongly fillable contact manifolds in all
dimensions.
(Many such examples exist in dimension~$3$, for instance
\cites{Giroux_plusOuMoins, Eliashberg3Torus}.
In dimension~$5$, the only ones known so far can be found in
\cite{WeakFillabilityHigherDimension}).

\begin{remark}
  Example~\ref{example: bourgeois from non-trivial open book not
    subcritical} generalizes in the following way to toric contact
  manifolds:
  Recall that there is an important difference between contact
  $5$-manifolds that have a torus action that is free and those where
  the $\TT^3$-action is not free, see \cite{Lerman1}.
  The open book on $(\SS^3, \xi_0)$ with page diffeomorphic to
  $T^*\SS^1$ can be obtained by the map $f(z_1,z_2) = z_1^2+z_2^2$,
  see Example~\ref{example: contact open book decomposition for
    sphere}.(b) and \ref{example: from open book to abstract open
    book}.(b).
  Both $\xi_0$ and $f$ are invariant under the free circle action on
  $\SS^3$ given by the multiplication with the matrices
  $\begin{pmatrix}
    \cos s & \sin s \\
    -\sin s & \cos s
  \end{pmatrix}$, which implies that not only the contact structure
  but also the open book is preserved by this action.
  Restricting the circle action to a cyclic subgroup
  $\ZZ_k\subset \SS^1$, we can quotient $\SS^3$ and obtain a lens
  space~$L_k$ carrying the natural contact structure, the induced open
  book decomposition, and a free circle action.
  The page of these open books is still diffeomorphic to $T^*\SS^1$,
  and its $0$-section is a Legendrian submanifold of $L_k$.
  A Bourgeois contact structure on $V\times \TT^2$ is clearly
  invariant under the obvious $\TT^2$-action, see
  Definition~\ref{def:BourgeoisStruct}, and with $V = \SS^3$ or
  $V = L_k$ as above, it is easy to verify that the initial circle
  action adds up to give a free $\TT^3$-action on $V\times \TT^2$.
  With some careful considerations, one obtains that all contact toric
  $5$-manifolds with a free $\TT^3$-action are either equivariantly
  contactomorphic to the unit cotangent bundle of $\TT^3$ or to one of
  the manifolds above.
  Thus according to Corollary~\ref{cor: page contains Legendrian then
    Bourgeois not subcritical}, none of the contact toric
  $5$-manifolds with a free $\TT^3$-action is subcritically fillable.
\end{remark}

\vspace{0.5cm}

In the appendix we review the one-to-one correspondence between
contact open book decompositions and abstract Liouville open books.
For this we use the language of ideal Liouville domains that has been
created for this purpose by Emmanuel Giroux \cite{GirouxIdealDomains}.
This language requires an initial investment of effort, but provides a
suitable framework for discussing the uniqueness of the resulting
contact structures up to homotopy and also for addressing problems
related to the structure along the binding.

\subsection*{Acknowledgments}
The authors would like to thank Paolo Ghiggini for his help in the
initial stage of the project and Stepan Orevkov for his help with the
proof of Lemma~\ref{lemma: exact pre-Lagrangians and positive loops}.
We also thank Jonathan Bowden for pointing out topological
obstructions to Stein fillability in high dimensions, and Patrick
Massot for many suggestions and improvements.
Samuel Lisi was supported in part by a University of Mississippi 
College of Liberal Arts Faculty Grant.
Aleksandra Marinković is partially supported by Ministry of Education and Science 
of Republic of Serbia, project ON174034.
Klaus Niederkrüger has been supported up to 2016 by the ERC Advanced
Grant LDTBud, and during 2017 by the Programme Avenir Saint-Etienne of
the Université de Lyon within the program ``Investissements d'Avenir''
(ANR-11-IDEX-0007).

\section{The Bourgeois contact structure}\label{sec: bourgeois
  structures introduction}

Given a contact manifold~$(V,\xi)$ and a symplectic
manifold~$(\Sigma,\omega)$, one naively obvious idea of how to
construct a contact structure on $V \times \Sigma$ would be to start
with the hyperplane field~$\xi \oplus T\Sigma$ which is an almost
contact structure, and try to deform it to a genuine contact
structure.
Bourgeois \cite{BourgeoisTori} succeeded in carrying this out in the
special case of $\Sigma=\TT^2$, using an open book decomposition of
$(V, \xi)$ as the input to the construction.
Gironella \cite{GironellaBourgeoisStructs} put this construction in a
more natural geometric framework in which open books appear
organically, generalizing the definition to oriented
surfaces~$\Sigma$.
We will describe Bourgeois's construction, reformulating it using the
notion of ideal Liouville domains and comparing with Gironella's more
general framework.
Appendix~\ref{appendix} provides background for a reader who might be
unfamiliar with the language of ideal Liouville domains.
Let $(V,\xi)$ be a contact manifold with a contact open book
decomposition~$(K,\vartheta)$ (See Definitions~\ref{def: smooth open
  book} and \ref{def: contact open book classical}).
From the appendix (see
Proposition~\ref{prop:contact_open_book_is_Liouville_ob}, we can
choose a contact form~$\alpha_V$ for $\xi$ and a function
$f= f_x + if_y\colon V\to \CC$ with $\vartheta = f/\abs{f}$ such that
$d(\alpha_V/\abs{f})$ defines an ideal Liouville structure
(Definition~\ref{def:ideal Liouville domain}) on every page of the
open book.
Clearly, the data~$\alpha_V$ and $f$ encode the contact structure and
the open book.
Accordingly, we call such $(\alpha_V, f)$ a \defin{representation} of
the contact open book.
(See Lemma~\ref{lem:Liouville_OB_is_contact_OB} for a justification of
this definition.)
Here and in the following, it will often be convenient to write
$f = f_x + i f_y = \rho\, \e^{i\vartheta}$.
We will also consider the $1$-form~$d\vartheta$ obtained from a
map~$\vartheta\colon V\setminus K \to \SS^1$ taking $\SS^1$ to be the
unit circle in $\CC$, which we also identify with $\RR/2\pi\ZZ$.
Strictly speaking, in what we write, $d\vartheta$ denotes the
differential of the argument of $\vartheta$ but we hope that this
abuse of notation does not cause any confusion.

\begin{definition}\label{def:BourgeoisStruct}
  The \defin{Bourgeois contact structure} associated to a
  representation $(\alpha_V, f)$ of the contact open book
  $(K, \vartheta)$ on $(V, \xi)$ and the standard orientation
  $d\varphi_1 \wedge d\varphi_2$ of $\TT^2$ is given by the kernel of
  the $1$-form
  \begin{equation*}
    \alpha = \alpha_V + f_x\,d\varphi_1 - f_y\,d\varphi_2
  \end{equation*}
  on $V\times \TT^2$, where $(\varphi_1, \varphi_2)$ denotes the
  standard coordinates on $\TT^2$.
\end{definition}

That a Bourgeois contact structure really is a contact structure will
follow directly from the more general result
Lemma~\ref{lem:characterization_BG_struct} below.
Notice that for a given open book decomposition, $(K, \vartheta)$, the
space of all choices of possible representations $(\alpha_V, f)$ is
contractible (this is discussed further in the appendix, see
Table~\ref{table: relationship between types of open book} and
following, also see \cite{GirouxIdealDomains}).
In particular then, a choice of contact open book determines an
isotopy class of contact structures on $V \times \TT^2$.
{It is also easy to convince oneself that up to contactomorphism
  the Bourgeois construction does not depend on the chosen
  identification of $\TT^2$ with $\SS^1\times \SS^1$.}
Gironella \cite{GironellaBourgeoisStructs} has extended this
definition of Bourgeois contact structure as a class of hyperplane
fields~$\tilde \xi_t$ that are deformations of a flat contact fiber
bundle~$\tilde \xi_0$ over $\Sigma$.
The hyperplane fields $\tilde \xi_t$ are contact for $t > 0$.
In this paper, we only consider the product case
$V \times \Sigma \to \Sigma$, where $\Sigma$ is a closed oriented
surface, and take the initial flat contact bundle to be
$\tilde \xi_0 = \xi \oplus T\Sigma$.
For deformations of these trivial bundles, Gironella additionally
provides a description in more elementary terms, which we repeat here.

\begin{definition}\label{def:BGstruct}
  Let $(V, \xi)$ be a contact manifold and $\Sigma$ be a closed
  oriented surface.
  A \defin{Bourgeois-Gironella contact structure~$\tilde \xi$} on
  $V \times \Sigma$ that deforms the flat contact
  bundle~$\xi \oplus T \Sigma$ is any contact structure that can be
  written as $\tilde \xi = \ker \alpha$ with
  \begin{equation*}
    \alpha = \alpha_V + \beta
  \end{equation*}
  where:
  \begin{enumerate}[(i)]
  \item $\alpha_V$ is a contact form on $V$ defining $\xi$;
    \label{enum:BGstruct i}
  \item $\beta$ is a $1$-form on $V\times \Sigma$ that vanishes on
    vectors that are tangent to the fibers~$V\times \{z\}$ for any
    $z\in \Sigma$; \label{enum:BGstruct ii}
  \item for each fixed $p \in V$, the restriction of $\alpha$ (or,
    equivalently, of $\beta$) to the slice $\{p\}\times \Sigma$ is a
    closed form. \label{enum:BGstruct iii}
  \item the orientation {induced on $V\times \Sigma$ by $\alpha$
      is the same one as the product orientation of $V$ with
      $\Sigma$}. \label{enum:BGstruct iv}
  \end{enumerate}
\end{definition}

Conditions~\eqref{enum:BGstruct i} and \eqref{enum:BGstruct ii} are
from \cite{GironellaBourgeoisStructs}*{Proposition 7.1} and
condition~\eqref{enum:BGstruct iii} is from
\cite{GironellaBourgeoisStructs}*{Claim 7.4}.
Note that in the cited reference, this formulation is given in the
case $\Sigma = \TT^2$.
These properties are local in $\Sigma$, however, so they remain
applicable in this seemingly more general case.
We do not know of examples of such structures for $\Sigma \ne \TT^2$,
however.

{
\begin{remark}\label{rmk: modifications of Bourgeois structure}
  Let $d\vol$ denote a volume form on $\Sigma$ compatible with the
  choice of orientation.
  The Bourgeois-Gironella structure~$(\ker \alpha, d\alpha)$ is
  homotopic to $(\xi\oplus T\Sigma, d\alpha_V + d\vol)$ as almost
  contact structures.
  This is verified by introducing an $\epsilon$-factor in
  the definition of $\alpha$
  \begin{equation*}
    \alpha_\epsilon \coloneqq \alpha_V + \epsilon\, \beta \;,
  \end{equation*}
  allowing us to deform $\ker \alpha$ to the flat
  contact bundle~$\ker \alpha_0 = \xi \oplus T\Sigma$.
  We then expand
  \begin{equation*}
    \begin{split}
      \alpha_\epsilon\wedge (d\alpha_\epsilon + \delta\, d\vol)^{n+1}
      &= (\alpha_V + \epsilon\, \beta)\wedge (d\alpha_V +
      \epsilon\,d\beta
      + \delta\, d\vol)^{n+1} \\
      &= \alpha_\epsilon \wedge (d\alpha_\epsilon)^{n+1} + (n+1)\,
      \delta \, \alpha_V\wedge (d\alpha_V)^n\wedge d\vol \;.
    \end{split}
  \end{equation*}
  Using that the restriction of $d\beta$ vanishes on every surface
  slice~$\{p\}\times \Sigma$, we check that the first term
  reduces to
  \begin{equation}\label{eqn:explicit_volume_form_expression_with_beta}
    \alpha_\epsilon \wedge (d\alpha_\epsilon)^{n+1} =
    \epsilon^2\, (n+1)\,\bigl( \frac{n}{2}\, d\beta^2\wedge \alpha_V \wedge
    d\alpha_V^{n-1} + \beta\wedge d\beta\wedge d\alpha_V^n\bigr)
    = \epsilon^2\, \alpha\wedge (d\alpha)^{n+1}
  \end{equation}
  so that if $\ker\alpha$ is contact any of the $\ker \alpha_t$ for
  $t> 0$ will be contactomorphic to it.
  Furthermore using \eqref{enum:BGstruct iv} from
  Definition~\ref{def:BGstruct}, we obtain that
  $\alpha_\epsilon\wedge (d\alpha_\epsilon + \delta\, d\vol)^{n+1}$ is
  strictly positive for all $\epsilon \ge 0$ and $\delta \ge 0$ as
  long as $\epsilon$ and $\delta$ do not vanish simultaneously.
  This shows that $\ker \alpha$ is indeed homotopic to
  $\xi\oplus T\Sigma$ as almost contact structures.
\end{remark}}

In the special case of $\Sigma = \TT^2$ and $\alpha$ a Bourgeois
contact form (i.e. so the coefficients of $\beta$ are
$\TT^2$-independent), if $\epsilon < 0$, we have an explicit
contactomorphism from $\alpha_\epsilon$ to $\alpha_{\abs{\epsilon}}$
by applying the orientation-preserving diffeomorphism
$(p; \varphi_1,\varphi_2) \mapsto (p; -\varphi_1, -\varphi_2)$ to
$V\times \TT^2$.
Now, essentially as an application of
\cite{GironellaBourgeoisStructs}*{Proposition~6.9}, we obtain the
following characterization of Bourgeois-Gironella structures deforming
the flat bundle $\xi \oplus T \Sigma$.
For the benefit of the reader, we provide a self-contained proof.
\begin{lemma}\label{lem:characterization_BG_struct}
  Let $\alpha = \alpha_V + \beta$ be a $1$-form on $V \times \Sigma$,
  where $\alpha_V$, and $\beta$ satisfy the
  conditions~\eqref{enum:BGstruct i} to \eqref{enum:BGstruct iv} of
  Definition~\ref{def:BGstruct} above.
  Suppose also that $V$ is of dimension at least~$3$.
  If $U$ is any positively oriented chart of $\Sigma$ with
  coordinates~$(\varphi_1, \varphi_2)$, then we can write $\alpha$ on
  $V\times U$ as
  \begin{equation*}
    \restricted{\alpha}{V\times U} =
    \alpha_V +f_x\,d\varphi_1 - f_y\, d\varphi_2
  \end{equation*}
  where $f = f_x + i f_y \colon V \times U \to \CC$ is a smooth
  function.
  The following two statements are then equivalent:
  \begin{itemize}
  \item [(a)]  $\alpha$ is a contact form;
  \item [(b)] for every chart~$U$ of $\Sigma$ and every point
    $(\varphi_1,\varphi_2) \in U$, the pair
    $\bigl(\alpha_V, f(\cdot\,, \varphi_1, \varphi_2)\bigr)$ with $f$
    as above is a representation of a contact open book on
    $(V,\xi)$.
  \end{itemize}
\end{lemma}
\begin{proof}
  Let $2n+1$ be the dimension of $V$.
  From \eqref{eqn:explicit_volume_form_expression_with_beta}, we know
  that the contact condition of $\alpha = \alpha_V + \beta$ is given
  by
  \begin{equation*}
    \alpha \wedge (d\alpha)^{n+1} =
    (n+1)\,\bigl( \frac{n}{2}\, d\beta^2\wedge \alpha_V \wedge
    d\alpha_V^{n-1} + \beta\wedge d\beta\wedge d\alpha_V^n\bigr) \ne 0 \;.
  \end{equation*}
  Now replacing $\beta$ by its representation in a chart,
  $f_x\,d\varphi_1 - f_y\,d\varphi_2$, and writing
  $f_x + i f_y = f = \rho\, \e^{i\vartheta}$, we obtain in these polar
  coordinates
  \begin{equation*}
    \beta \wedge d\beta = \rho^2 \, d_V\vartheta \wedge d\varphi_1
    \wedge d\varphi_2 \quad\text{ and }\quad
    d\beta^2 = 2 \rho\, d_V\rho \wedge d_V \vartheta \wedge d\varphi_1 \wedge
    d\varphi_2 
  \end{equation*}
  where $d_V$ is the exterior derivative only in $V$-direction.
  It therefore follows that
  \begin{equation}\label{eqn: expansion of BG contact form}
    \begin{split}
      \alpha \wedge (d\alpha)^{n+1} &= { (n+1)\, \bigl[ n\,
        d_Vf_x\wedge
        d_V f_y \wedge \alpha_V\wedge (d\alpha_V)^{n-1} }\\
      & \qquad \qquad\qquad \qquad {+ (f_x\, d_Vf_y - f_y\,
        d_Vf_x) \wedge
        (d\alpha_V)^n \bigr] \wedge d\varphi_1\wedge d\varphi_2 }\\
      &= (n+1)\, \bigl[ n\, \rho\, d_V\rho\wedge d_V\vartheta \wedge
      \alpha_V\wedge (d\alpha_V)^{n-1} + \rho^2\, d_V\vartheta \wedge
      (d\alpha_V)^n \bigr] \wedge d\varphi_1\wedge d\varphi_2 \; .
    \end{split}
  \end{equation}
  First, observe that if
  $\bigl(\alpha_V, f(\cdot, \varphi_1, \varphi_2)\bigr)$ is a
  representation of a contact open book, then $\alpha$ is a contact
  form because the term in brackets agrees with the
  expansion~\eqref{eq: volume form from open book} of the volume
  form~$\Omega_V$ on $V$ in Lemma~\ref{lemma: volume form from open
    book}.
  Thus as we wanted to show
  $\alpha\wedge (d\alpha)^{n+1} = (n+1)\, \Omega_V\wedge
  d\varphi_1\wedge d\varphi_2$ does not vanish.
  To prove the converse, we now suppose instead that $\alpha$ is a
  contact form.
  Fix a point $z \in \Sigma$.
  We must now prove the following statements:
  \begin{enumerate}[(i)]
  \item $0 \in \CC$ is a regular value of $p \mapsto f(p,z)$;
  \item $K \coloneqq \{ p \in V \, | \, f(p,z) = 0 \}$ is non-empty;
  \item $\vartheta := f/\abs{f}\colon V \setminus K\to \SS^1$ is a
    fibration.
  \item $d\bigl( \alpha_V/\abs{f}\bigr)$ restricts to each fiber
    $\vartheta = \vartheta_0$ as an ideal Liouville structure.
  \end{enumerate}
  By Remark~\ref{rmk: top open book specified by complex function},
  the first three properties give that $f(\cdot, z)$ defines an open
  book decomposition on $V$.
  The fourth gives that it is a contact open book (see
  Lemma~\ref{lem:Liouville_OB_is_contact_OB} for details).
  To prove the first statement, let $p \in V$ be such that
  $f(p, z) = 0$.
  Since $\alpha \wedge d\alpha^{n+1}$ is a volume form by assumption,
  it follows from the first line of Equation~\eqref{eqn: expansion of
    BG contact form} that $d_V f_x \wedge d_V f_y$ does not vanish at
  $p$ so that $0$ is a regular value of $p \mapsto f(p, z)$.
  For the third statement, we observe that by combining the fact that
  $\alpha \wedge d\alpha^{n+1}$ is a volume form with the second line
  of Equation~\eqref{eqn: expansion of BG contact form},
  $d_V\vartheta$ cannot vanish on $V \setminus K$.
  Hence, $\vartheta(\cdot, z) \colon V \setminus K \to \SS^1$ is a
  submersion.
  In order to show that $K$ is non-empty and also to show that
  $d\bigl(\alpha_V/\abs{f} \bigr)$ restricts to the fibers as an ideal
  Liouville structure, we compute
  \begin{equation*}
    \Bigl( d_V \bigl( \tfrac{1}{\rho}\, \alpha_V \bigr) \Bigr)^n
    = \frac{1}{\rho^{n+2}}\, \bigl( 
    - n \rho \, d_V\rho \wedge \alpha_V \wedge (d\alpha_V)^{n-1}
    + \rho^2 d\alpha_V^n \bigr) \;.
  \end{equation*}
  Observe that Equation~\eqref{eqn: expansion of BG contact form} can
  be rearranged to obtain:
  \begin{equation*}
    \begin{split}
      \alpha \wedge (d\alpha)^{n+1} &= (n+1)\, \bigl[
      -n\, \rho\, d_V\rho\wedge \alpha_V\wedge (d\alpha_V)^{n-1} + \rho^2\, (d\alpha_V)^n \bigr] \wedge d_V\vartheta \wedge d\varphi_1\wedge d\varphi_2 \\
      &= (n+1)\, \rho^{n+2}\, \Bigl( d_V \bigl( \tfrac{1}{\rho}\,
      \alpha_V \bigr) \Bigr)^n \wedge d_V \vartheta \wedge d \varphi_1
      \wedge d\varphi_2 \;.
    \end{split}
  \end{equation*}
  This is a volume form by assumption, so it follows that
  $d\bigl(\alpha_V/\abs{f}\bigr)$ is symplectic when restricted to a
  fiber~$\vartheta = \vartheta_0$.
  For the sake of contradiction, suppose that $K$ is empty.
  In that case, the fiber
  $\{ p \in V \, | \, \vartheta(p,z) = \vartheta_0 \}$ is a closed
  submanifold of $V$ of dimension~$2n$.
  The restriction of $d\bigl(\alpha_V/\rho\bigr)$ to this submanifold
  is symplectic.
  By Stokes' theorem, this is only possible if the dimension of $V$ is
  $1$.
  Thus, $K$ is non-empty for $2n+1 \ge 3$.
  Having established that $K$ is non-empty, it follows that
  $d\bigl(\alpha_V/\rho\bigr)$ is an ideal Liouville domain structure
  on the closure of $\vartheta^{-1}(\vartheta_0)$.
  This then shows that $(\alpha_V, f)$ is a representation of a
  contact open book on $V$, as required.
\end{proof}

It follows in particular from this lemma that Bourgeois contact
structures as given by Definition~\ref{def:BourgeoisStruct} really are
contact structures.
In fact, Bourgeois structures are the special
Bourgeois-Gironella contact structures on $V\times \TT^2$ that are
invariant under the canonical torus action.
Gironella shows the non-obvious fact
\cite{GironellaBourgeoisStructs}*{Proposition 6.11} that the
$\TT^2$-average of any Bourgeois-Gironella contact form
$\alpha_V + \beta$ is also a contact form.
This averaging process gives us a canonical map from
Bourgeois-Gironella structures to Bourgeois structures.
We do not know of any example of a Bourgeois-Gironella contact form
that is not isotopic through Bourgeois-Gironella forms to its $\TT^2$
average.

\section{The Bourgeois structure for open books with inverted
  monodromy}\label{section: inverse monodromy}

In this section we will prove Theorem~\ref{thm: inverse monodromy same
  Bourgeois structure}.
To achieve this aim, we first describe an explicit modification of a
given contact structure supported by a contact open book.
The result of this construction will be a new contact structure that
is supported by an open book with identical pages and binding as the
first one, but with opposite coorientation.
We then show that the monodromies of the two open books are the
inverse of each other, and we conclude by studying how this
modification affects the Bourgeois construction.

\begin{lemma}\label{lemma: contact form inverse monodromy}
  Let $(V,\xi_+)$ be a contact manifold with a compatible open book
  decomposition~$(K,\vartheta)$, and let $\alpha_+$ be any contact
  form that is supported by this open book.
  The space of functions~$f = f_x + if_y\colon V \to \CC$ (writing
  $\abs{f}^2\, d\vartheta = f_x\,df_y - f_y \, df_x$) that satisfy the
  properties below is convex and non-empty
  \begin{itemize}
  \item [(i)] $(\alpha_+, f)$ is a representation of a Liouville open
    book on $(K,\vartheta)$ in the sense of
    Definition~\ref{definition:Liouville open book};
  \item [(ii)] the $1$-form
    \begin{equation*}
      \alpha_- \coloneqq \alpha_+ - C\, \abs{f}^2\, d\vartheta
    \end{equation*}
    is a contact form for every sufficiently large constant~$C \gg 1$.
  \end{itemize}
  The contact forms~$\alpha_+$ and $\alpha_-$ induce opposite
  orientations on $V$, $\alpha_-$ is adapted to the open book
  decomposition~$(K,\overline{\vartheta})$, and while its restriction
  to the binding and pages does not differ from the one of $\alpha_+$,
  the coorientation of pages and binding is reversed.
\end{lemma}
\begin{proof}
  Let $f = f_x + if_y$ and $g = g_x+ig_y$ be two functions that
  satisfy the two properties stated above.
  We know from the appendix that the set of functions~$F$ such that
  $(\alpha_+, F)$ is a representation forms a non-empty convex set, so
  let us concentrate on property~(ii).
  Define for a sufficiently large $C\gg 1$ the two contact forms
  \begin{equation*}
    \alpha_- \coloneqq \alpha_+ - C\, \abs{f}^2\,d\vartheta
    \quad\text{ and }\quad
    \beta_- \coloneqq \alpha_+ - C\, \abs{g}^2\,d\vartheta \;.
  \end{equation*}
  We need to show that the interpolation
  \begin{equation*}
    \alpha_s \coloneqq (1-s)\,\alpha_- + s\,\beta_-
  \end{equation*}
  satisfies for all $s\in [0,1]$ the contact property.
  Writing $\alpha_s$ as
  \begin{equation*}
    \alpha_s \coloneqq \alpha_+ - C\,
    \bigl((1-s)\,\abs{f}^2 + s\abs{g}^2\bigr)\, d\vartheta\;,
  \end{equation*}
  it is obvious that all terms of
  $\alpha_s\wedge \bigl(d\alpha_s\bigr)^n$ contain at most one
  $d\vartheta$-factor, and in particular $\abs{f}^2$- and
  $\abs{g}^2$-terms will never mix.
  The contact condition simplifies to
  \begin{equation*}
    \alpha_s \wedge \bigl(d\alpha_s\bigr)^n =
    (1-s)\,\alpha_-\wedge \bigl(d\alpha_-\bigr)^n
    + s\,\beta_- \wedge \bigl(\beta_-\bigr)^n \ne 0 \;,
  \end{equation*}
  which is true by assumption thus proving the desired convexity
  property.
  We still need to show that it is not empty.
  Let $f = f_x + if_y\colon V \to \CC$ be a function defining the open
  book, and write for simplicity $\rho = \abs{f}$ so that
  $f_x\, df_y - f_y \, df_x = \rho^2 d\vartheta$.
  The condition that
  $d\bigl(\alpha_+/\abs{f}\bigr) = d\bigl(\alpha_+/\rho\bigr)$ is a
  Liouville form on each page can be verified by computing
  \begin{equation*}
    \rho^{n+2}\, d\vartheta \wedge \bigl(d(\alpha_+/\rho)\bigr)^n
    = \rho^2\, d\vartheta \wedge \bigl(d\alpha_+\bigr)^n + n\rho\,
    d\rho\wedge d\vartheta \wedge \alpha_+ \wedge
    \bigl(d\alpha_+\bigr)^{n-1} \ne 0 \;,
  \end{equation*}
  and the condition that $\alpha_-$ is a contact form is verified by
  computing
  \begin{equation*}
    \alpha_-\wedge d\alpha_-^n  = \alpha_+ \wedge (d\alpha_+)^n
    - C\, \bigl[\rho^2\,d\vartheta \wedge (d\alpha_+)^n +
    2n\, \rho\,d\rho\wedge d\vartheta \wedge
    \alpha_+\wedge (d\alpha_+ )^{n-1} \bigr] \;.
  \end{equation*}
  In both cases, the term $\rho^2\, d\vartheta\wedge (d\alpha_+)^n$ is
  never negative and only vanishes along the binding.
  The second term
  $\rho\,d\rho\wedge d\vartheta \wedge \alpha_+\wedge
  (d\alpha_+)^{n-1}$
  can be understood as follows: Along the binding the term is
  positive, since $\rho\,d\rho\wedge d\vartheta$ is an area form on
  the disk and because the restriction of $\alpha_+$ to the binding is
  by assumption a positive contact form.
  If $\rho$ is a function that increases linearly in radial direction
  at the binding $K$ and that is constant outside a sufficiently
  small neighborhood {of $K$}, 
  then it follows that
  $\rho\,d\rho\wedge d\vartheta$ is positive along the binding and
  everywhere else is non-negative.
  This shows that the function~$\rho$ can be chosen in such a way that
  $(\alpha_+, f)$ is a representation and such that $\alpha_-$ will be
  for any sufficiently large $C$ a contact form.
  It remains to show that $\xi_- = \ker \alpha_-$ is supported by
  $(K,\overline{\vartheta})$.
  For this note that $(K,\overline{\vartheta})$ and $(K,\vartheta)$
  have the same pages and binding.  The restriction of $\alpha_-$ and
  $\alpha_+$ agree on both subsets, since the additional term vanishes
  when restricted to either.
  The contact forms $\alpha_+$ and $\alpha_-$ induce opposite
  orientations on $V$, which is compatible with the choice of
  coorientations given by $\vartheta$ and $\overline{\vartheta}$
  respectively.
\end{proof}

\begin{lemma}\label{lemma: constructed form has inverse monodromy}
  Assume we are in the setup of the previous lemma.
  The abstract Liouville open books corresponding to $\alpha_-$ and
  $(K,\overline{\vartheta})$ and to $\alpha_+$ and $(K,\vartheta)$
  have identical ideal Liouville domains as pages, but their
  monodromies are the inverse of each other.
\end{lemma}
\begin{proof}
  Let $f = f_x + if_y$ be the function used in the previous lemma.
  It is easy to check that $\overline f = f_x - if_y$ is a function
  defining $(K,\overline{\vartheta})$ and since the restrictions of
  $\alpha_+$ and $\alpha_-$ agree on all pages, it is clear that
  $\alpha_- / \abs{\overline{f}} = \alpha_+ / \abs{f}$ defines on
  every page the same ideal Liouville structure as the initial open
  book.
  This shows that the pages of the abstract open book corresponding to
  $(K,\vartheta)$ with contact form~$\alpha_+$ and the ones
  corresponding to $(K,\overline{\vartheta})$ with contact
  form~$\alpha_-$ are identical as ideal Liouville domains.
  We set $\lambda_+ := \alpha_+/\abs{f}$, and
  $\lambda_- := \alpha_-/\abs{f}$.
  Recall that we recover the monodromy of the Liouville open book
  $(K,\vartheta, d\lambda_+)$ by following the flow of a spinning
  vector field from an initial page back to itself.
  By Lemma~\ref{lemma: exists vector field to recover monodromy}, we
  can specify a unique spinning vector field~$Y_+$ by the equations
  \begin{equation*}
    d\vartheta(Y_+) = 2\pi \quad\text{ and }\quad
    \iota_{Y_+} d\lambda_+ = 0 \;.
  \end{equation*}
  We claim that $Y_- = - Y_+$ is a spinning vector field for the
  Liouville open book $(K,\overline{\vartheta}, d\lambda_-)$.
  Clearly $Y_-$ vanishes along the binding and
  $d\overline{\vartheta}(Y_-) = +2\pi$.
  It only remains to show that the flow of $Y_-$ preserves the ideal
  Liouville structure on every page.
  For this simply compute
  \begin{equation*}
    \lie{Y_-} d\lambda_- = - \lie{Y_+} \bigl(d\lambda_+ - C\,
    d\bigl(\frac{\rho^2}{\abs{f}}\bigr)\wedge d\vartheta \bigr) =
    C\, \lie{Y_+} \bigl(d\rho \wedge d\vartheta\bigr)
    = C\,\bigl(\lie{Y_+} d\rho\bigr)
    \wedge d\vartheta \;.
  \end{equation*}
  Since $d\vartheta$ vanishes on every page, we see that $Y_-$ is
  indeed a spinning vector field for $d\lambda_-$.
   The time-$1$ flow of $Y_-$ is obviously the inverse of the
    time-$1$ flow of $Y_+$, thus we have shown that the corresponding
    abstract open books have the equal page and that the monodromies
    are the inverse of each other.
\end{proof}

We now show that inverting the monodromy of an open book has no
influence on the Bourgeois construction.

\begin{theorem_inverse_monodromy}
  Let $(V,\xi_+)$ and $(V,\xi_-)$ be closed contact manifolds
  supported by abstract Liouville open books that have the same page
  but inverse monodromy.
  Then the two corresponding Bourgeois structures on $V\times \TT^2$
  are contactomorphic.
\end{theorem_inverse_monodromy}

This is a corollary of
Lemma~\ref{lemma: constructed form has inverse monodromy} combined
with the following result.

\begin{lemma}
  Assume we are in the setup of Lemma~\ref{lemma: contact form inverse
    monodromy}, so that $(V, \xi_+)$ is a contact manifold with a
  compatible open book decomposition~$(K, \vartheta)$ that is
  represented by $(\alpha_+, f)$ and $\xi_- = \ker(\alpha_-)$ with
  $\alpha_- = \alpha_+ - C\, \abs{f}^2\, d\vartheta$ for sufficiently
  large $C$ is a contact structure on $V$ that is supported by the
  open book~$(K, \overline{\vartheta})$.
  Then, any Bourgeois contact structure on $V \times \TT^2$ associated
  to the contact open book $(\xi_+, K, \vartheta)$ and the standard
  orientation of $\TT^2$ is isotopic through contact structures to the
  Bourgeois contact structure on $V \times \TT^2$ associated to
  $(\xi_-, K, \overline{\vartheta})$ and the reversed orientation on
  $\TT^2$.
\end{lemma}
\begin{proof}
  With the notation as in Lemma~\ref{lemma: contact form inverse
    monodromy}, it follows that $(\alpha_-, \overline{f})$ is a
  representation of the open book $(\xi_-,K, \overline{\vartheta})$,
  where $\overline{f} = f_x - i f_y$ denotes the complex conjugate.
  From Definition~\ref{def:BourgeoisStruct}, the Bourgeois contact
  structure associated to $(\alpha_+, f)$ (and the standard
  orientation on $\TT^2$) is given by
  \begin{equation*}
    \alpha_+ + f_x\,d\varphi_1 - f_y\, d\varphi_2 \;.
  \end{equation*}
  Consider now the parametric family of $1$-forms given by
  \begin{equation*}
    \alpha_\tau = \alpha_+ + f_x \, d\varphi_1 - f_y \, d\varphi_2
    -\tau C \,\abs{f}^2 \, d\vartheta, \qquad 0 \le \tau \le 1 \;.
  \end{equation*}
  A direct computation shows that
  $\alpha_\tau \wedge (d\alpha_\tau)^{n+1} = \alpha_0 \wedge
  (d\alpha_0)^{n+1}$, and thus these are all contact forms.
  Very explicitly, we observe that
  $\alpha_\tau = \Phi_\tau^*\alpha_0$, with $\Phi_\tau$ given by
  \begin{align*}
    \Phi_\tau \colon V \times \TT^2 &\to V \times \TT^2 \\
    (p; \varphi_1, \varphi_2) &\mapsto (p;\, \varphi_1 - \tau C f_y, \,
                                \varphi_2 - \tau C f_x) \;.
  \end{align*}
  Now, observe that
  $\alpha_1 = \alpha_- + f_x \, d\varphi_1 - f_y \, d\varphi_2$, which
  is the Bourgeois form on $V \times \TT^2$ associated to the
  representation $(\alpha_-, \overline{f})$ and the orientation on
  $\TT^2$ given by $(\partial_{\varphi_1}, -\partial_{\varphi_2})$.
\end{proof}

Finally we obtain the desired contactomorphism for Theorem~\ref{thm:
  inverse monodromy same Bourgeois structure} by composing the isotopy
from the previous lemma with the diffeomorphism
$(p;\varphi_1,\varphi_2) \mapsto (p;\varphi_1,-\varphi_2)$ on
$V\times \TT^2$.

\section{Explicit constructions of fillings}\label{sec: explicit
  fillings}

In this section, we will prove Theorem~\ref{thm: fillings for V that
  imply fillings for VxT2} from the introduction.
Let $(V,\xi)$ be a contact manifold, and let $\alpha_V$ be a contact
form for $\xi$.
A symplectic manifold~$(W,\omega)$ is called a \defin{weak filling} of
$(V,\xi)$ {(see \cite{WeakFillabilityHigherDimension})}, if $W$
is compact with (oriented) boundary~$\p W = V$, and if for every
$T\in [0,\infty)$
\begin{equation*}
  \alpha_V \wedge (T\,d \alpha_V + \omega)^n > 0 \;,
\end{equation*}
where $\dim V = 2n+1$.
The following argument was inspired by a $3$-dimensional proof in
\cite{Giroux_plusOuMoins}, and has been sketched in
\cite{WeakFillabilityHigherDimension}*{Example~1.1}.
A proof mostly identical to ours has recently appeared in
\cite{GironellaBourgeoisStructs}, but since the argument is relatively
short we prefer to restate it here for completeness of our
presentation.

\begin{theorem_weak_filling}
  Let $(V,\xi)$ be a contact manifold that is weakly filled by
  $(W,\omega)$, and let $(K,\vartheta)$ be any open book that is
  compatible with $\xi$.
  Then the associated Bourgeois contact structure on $V\times \TT^2$
  is isotopic to a contact structure that can be weakly filled by
  $(W\times \TT^2, \omega\oplus \vol_{\TT^2})$.
\end{theorem_weak_filling}
\begin{proof}
  Using the modified Bourgeois contact form~$\alpha_\epsilon$ from
  Remark~\ref{rmk: modifications of Bourgeois structure}, we obtain by
  \begin{equation*}
    P_\epsilon(T) \coloneqq \alpha_\epsilon\wedge
    \bigl(T\, d\alpha_\epsilon + \omega + \vol_{\TT^2}\bigr)^{n+1}
  \end{equation*}
  a family of polynomials of degree at most $n+1$ in $T$ with
  coefficients in $\Omega^{2n+1}(V\times \TT^2)$ that depend smoothly
  on $\epsilon$.
  We will show that if $\epsilon> 0$ is chosen sufficiently small,
  then $P_\epsilon(T)$ will be positive for every $T\in [0,\infty)$, 
  so that $(W\times \TT^2, \omega\oplus \vol_{\TT^2})$ is a
  weak filling of $\ker\alpha_\epsilon$, which by Remark~\ref{rmk:
    modifications of Bourgeois structure} is isotopic to $\ker\alpha$.
  First note that the leading term of $P_\epsilon$ is
  $\epsilon^2\,\alpha\wedge d\alpha^{n+1}\, T^{n+1}$.
  Its coefficient vanishes for $\epsilon = 0$, but is strictly
  positive for $\epsilon \ne 0$.
  For  $\epsilon = 0$, we compute
  \begin{equation*}
    P_0(T)  = \alpha_V\wedge \bigl(T\, d\alpha_V +
    \omega + \vol_{\TT^2}\bigr)^{n+1} 
    = (n+1)\,\alpha_V\wedge \bigl(T\, d\alpha_V + \omega)^n\wedge
    \vol_{\TT^2} \;.
  \end{equation*}
  This form is strictly positive for all $T\in[0,\infty)$ by the
  assumption that $(W,\omega)$ is a weak filling of $(V,\xi)$.
  Furthermore we see that $P_0(T)$ is of degree~$n$ in $T$ with a
  strictly positive coefficient for the leading term.
  Any small perturbation of $P_0$ \emph{inside the polynomials of
    degree~$n$} will also be strictly positive on $T\in[0,\infty)$:
  If we choose a sufficiently large~$T_0$, the leading term of
  $P_0(T)$ dominates 
  the remaining terms of the
  polynomial
  for $T > T_0$.
  Thus none of the polynomials of degree~$n$ that are close to $P_0$
  will vanish for $T > T_0$.
  On the other hand, if we only consider a compact interval~$[0,T_0]$,
  it follows by continuity that a small perturbation of $P_0$ (even in
  the space of continuous functions) cannot vanish on $[0,T_0]$
  either.
  Combining this with the positivity of the coefficient for
  $T^{n+1}$-term in $P_\epsilon$ we obtain the desired result.
\end{proof}

\vspace{0.5cm}

Before proving part~(b) of Theorem~\ref{thm: fillings for V that imply
  fillings for VxT2}, we will briefly recall the basic definitions on
Weinstein manifolds.
A \defin{Weinstein manifold} $(W, \omega, X, f)$ is a symplectic
manifold $(W,\omega)$ without boundary, together with
\begin{itemize}
\item [(i)] a complete vector field~$X$ such that
  $\lie{X}\omega = \omega$, a so-called complete \defin{Liouville
    vector field}, and
\item [(ii)] a proper Morse function $f\colon W\to [0,\infty)$ that is
  a Lyapunov function for $X$, meaning that there is a positive
  constant~$\delta$ such that
  $df(X) \ge \delta\cdot (\norm{X}^2 + \norm{df}^2)$ with respect to
  some Riemannian metric.
\end{itemize}
Other definitions may not require $f$ to be a Morse function, but we
follow \cite{CieliebakEliashberg} and just note that a given Weinstein
manifold is symplectomorphic to one whose Lyapunov function is Morse.
The topology of a Weinstein manifold~$(W, \omega, X, f)$ is relatively
restricted, because the index of every critical point of $f$ is less
than or equal to half the dimension of $W$.
If $f$ has only critical points of index strictly less than
$\frac{1}{2} \dim W$, then we say that $(W, \omega, X, f)$ is a
\defin{subcritical Weinstein manifold}; and if $f$ has only critical
points of index not more than $\frac{1}{2} \dim W - k$, then we say
that $(W, \omega, X, f)$ is \defin{$k$-subcritical}.
The complex plane with the standard symplectic
form~$\omega_0 = dx\wedge dy$, the Liouville vector
field~$X_0 = \frac{1}{2}\,(x\,\partial_x + y\, \partial_y)$, and Morse
function $f_0(x+ i y) = x^2 + y^2$ is a Weinstein manifold.
The \defin{stabilization of a Weinstein manifold~$(W,\omega, X,f)$} is
the product Weinstein manifold
$(W\times \CC, \omega \oplus \omega_0, X \oplus X_0, f + f_0)$.
The stabilization of any Weinstein manifold is subcritical, and
according to the following result by Cieliebak
\citelist{\cite{Cieliebak_subcritStein_split}
  \cite{CieliebakEliashberg}*{Section~14.4}}, subcritical Weinstein
manifolds are essentially stabilizations.

\begin{theorem}[Cieliebak]\label{thm: Cieliebak}
  Every subcritical Weinstein manifold of dimension~$2n$ is
  symplectomorphic to the stabilization of a Weinstein manifold of
  dimension~$2n-2$.
\end{theorem}

Note that we only use this theorem to obtain that any compact
Lagrangian in a subcritical Weinstein domain is Hamiltonian
displaceable.
This simpler result is given by
\cite{BiranCieliebak_subcritical_Lagrangian}*{Lemma~3.2}.
A regular level set~$M_c = f^{-1}(c)$ of a Weinstein
manifold~$(W,\omega, X,f)$ carries a natural contact structure given
by the kernel of the $1$-form
$\alpha_c \coloneqq \restricted{\omega(X,\cdot)}{T M_c}$.
We say that a contact manifold~$(V,\xi)$ is \defin{(subcritically)
  Weinstein fillable}, if it is contactomorphic to a regular level
set~$(M_c, \ker \alpha_c)$ of a (subcritical) Weinstein
manifold~$(W,\omega, X,f)$ such that all critical values of $f$ are
strictly smaller than $c$.
Let $(V,\xi)$ be a contact manifold that is subcritically filled by a
stabilized Weinstein manifold
$(W\times \CC, \omega \oplus \omega_0, X \oplus X_0, f + f_0)$.
A computation shows that $(V, \xi)$ is supported by the open book with
binding~$K_0 = V \cap \bigl(W\times \{0\}\bigr)$ and fibration
$\vartheta_0\colon V\setminus K_0 \to \SS^1, (p,z) \mapsto z/\abs{z}$.
{The corresponding abstract open book has page~$W$ and trivial
  monodromy.}
The details of this are carried out in Example~\ref{example:trivial
  monodromy}.
The following proposition finishes the proof of Theorem~\ref{thm:
  fillings for V that imply fillings for VxT2}.
It shows that certain Bourgeois contact structures are Weinstein
fillable.

\begin{theorem_subcrit_filling}
  Let $(V,\xi)$ be a closed contact manifold that is subcritically
  filled by the Weinstein manifold
  $(W\times \CC, \omega \oplus \omega_0, X \oplus X_0, f + f_0)$.
  Let $(K_0, \vartheta_0)$ be the associated open book with trivial
  monodromy.
  The Bourgeois contact structure on $V\times \TT^2$ obtained by using
  the contact open book~$(K_0, \vartheta_0)$ can be filled by the
  Weinstein manifold
  \begin{equation*}
    \bigl(W\times T^*\TT^2, \omega \oplus d\lcan, X \oplus X_{\TT^2},
    f + f_{\TT^2}\bigr) \;,
  \end{equation*}
  where the cotangent bundle of $\TT^2$ is written with coordinates
  $(q_1,q_2; p_1,p_2)\in \TT^2 \times \RR^2$,
  $X_{\TT^2} = p_1\,\p_{p_1} + p_2\,\p_{p_2}$, and
  $f_{\TT^2} = p_1^2 + p_2^2$.
\end{theorem_subcrit_filling}
\begin{proof}
  Identify $(V,\xi)$ with the regular level set~$M_c$ in
  $\bigl(W\times \CC, J\oplus i\bigr)$.
  The Bourgeois structure on $M_c\times \TT^2$ is given by the contact
  form
  \begin{equation*}
    \alpha =  \lambda_W + x\, dy - y\, dx +  x\,d\varphi_1 - y\,d\varphi_2
  \end{equation*}
  where $\lambda_W$ is the Liouville form~$\iota_X\omega$ on $W$,
  $z=x+iy$ are the coordinates on $\CC$, and $(\varphi_1,\varphi_2)$
  are the coordinates on the torus.
  The diffeomorphism from $W\times \CC\times \TT^2$ to
  $W\times T^*\TT^2$ that sends
  $(x,y;\varphi_1,\varphi_2) \in \CC\times \TT^2$ to
  $(q_1,q_2;p_1,p_2) = (-\varphi_1 - y,\varphi_2 + x; x,y) \in
  T^*\TT^2$
  and keeps the $W$-factor unchanged is the desired contactomorphism.
  Note in particular that it pulls back $f+f_{\TT^2}$ to $f + f_0$.
\end{proof}

\section{Obstructions to subcritical fillings}\label{sec: subcritical
  fillings}

The aim of this section is to show that most Bourgeois structures are
not subcritically fillable.
We will first introduce the necessary preliminaries to prove
Theorem~\ref{thm: main obstruction result}.
Let $(V,\xi)$ be a contact manifold.

\begin{definition}
  A submanifold~$P$ of a contact manifold $(V,\xi)$ is called
  \defin{pre-Lagrangian}
  \begin{itemize}
  \item if $\dim P = \frac{1}{2}\,(\dim V + 1)$ and
  \item if there exists a contact form $\alpha$ for $\xi$ such that
    $\restricted{d\alpha}{TP} = 0$.
  \end{itemize}
  It is easy to see that $\xi$ induces a \emph{regular} Legendrian
  foliation on such a $P$.
\end{definition}

The symplectization~$(SV,d\lcan)$ of $(V,\xi)$ is the submanifold
\begin{equation*}
  SV \coloneqq \bigl\{(p,\eta_p) \in T^*V \bigm|\,
  \text{$\ker\,\eta_p = \xi_p$ and
    $\eta_p$ agrees with the coorientation of $\xi_p$}\bigr\}
\end{equation*}
of the cotangent bundle of $V$, where $\lcan$ denotes the restriction
of the canonical Liouville $1$-form of $T^*V$.
We denote by $\pi_V\colon SV\to V$ the projection $\pi_V(p,\eta_p) =
p$.
The choice of a contact form~$\alpha$ for $\xi$ allows us to identify
$(SV,d\lcan)$ with $\bigl(\RR\times V, d(e^t\alpha)\bigr)$ via the map
$(t,p) \in \RR\times V \mapsto e^t\alpha_p \in SV$ making use of the
tautological property $\beta^*\lcan = \beta$ for
$\beta \in \Omega^1(V)$.
An equivalent definition of $P\subset V$ being pre-Lagrangian is to
say that the symplectization contains a Lagrangian~$L\subset SV$ such
that the projection $\pi_V\colon SV\to V$ restricts to a
diffeomorphism $\restricted{\pi_V}{L}\colon L \to P$ (see
\cite{EliashbergHoferSalamon}*{Proposition~2.2.2}).
Every such Lagrangian is called a \defin{Lagrangian lift of $P$}.
These lifts are related to the choice of a contact form~$\alpha$ with
$\restricted{d\alpha}{TP} = 0$ by $L = \alpha(P)$, where $\alpha$ is
regarded as a section $V \to SV$.
Gromov calls a Lagrangian~$L$ in a symplectic manifold $(W,\omega)$
\defin{weakly exact} \cite{Gromov_HolCurves}*{2.3.B$_3$} if
$\int_{\DD^2}u^*\omega$ vanishes for every smooth map
\begin{equation*}
  u\colon (\DD^2,\p \DD^2)\to (W,L) \;.
\end{equation*}
In the spirit of
\cite{MassotNiederkrugerNonconnectedContactomorphisms} we use the same
notion for pre-Lagrangians:
a pre-Lagrangian~$P$ in a contact manifold $(V,\xi)$ is called
\defin{weakly exact} if for every contact form~$\alpha$ with
$\restricted{d\alpha}{TP} = 0$ and for every smooth map $u\colon
(\DD^2, \p \DD^2) \to (V,P)$, the integral $\int_{\DD^2}u^*d\alpha$
vanishes.
In fact, if this integral is zero for one such form, then it is zero
for every $\alpha'$ for which $\restricted{d\alpha'}{TP} = 0$.
In contrast to the Lagrangian case where weak exactness is a rather
subtle symplectic property, the weak exactness for pre-Lagrangians
reduces to the following topological observation:

\begin{lemma}\label{lemma: exact pre-Lagrangians and positive loops}
  A closed pre-Lagrangian~$P\subset (V,\xi)$ is weakly exact if and only if
  every smooth loop in $P$ that is positively transverse to the
  foliation $\fF \coloneqq \xi\cap TP$ is non-trivial in $\pi_1(V)$.
\end{lemma}

\begin{remark}
  In dimension~$3$, the only type of closed pre-Lagrangian is an
  embedded torus whose characteristic foliation is linear.
  In this case, Lemma~\ref{lemma: exact pre-Lagrangians and positive
    loops} states that weak exactness is equivalent to the
  incompressibility of the torus, because the {transverse loops
    generate the full fundamental group of the torus.}
  Tight contact manifolds with positive Giroux torsion contain
  ``many'' incompressible pre-Lagrangians and are at the same time not
  even strongly fillable.
  Theorem~\ref{thm: main obstruction result} requires the existence of
  only one incompressible pre-Lagrangian, but this weaker condition
  only contradicts a more specific type of filling.
\end{remark}

\begin{proof}[Proof of Lemma~\ref{lemma: exact pre-Lagrangians and
    positive loops}]
  Let $\alpha$ be a contact form on $V$ such that
  $\restricted{d\alpha}{TP} = 0$.
  Assume that $P$ is weakly exact and that $\gamma\subset P$ is a
  smooth loop that is positively transverse to $\fF$.
  If $[\gamma]$ were trivial in $\pi_1(V)$, we could choose a (smooth)
  map $u\colon(\DD^2,\p \DD^2) \to (V,P)$ with
  $\restricted{u}{\p \DD^2} = \gamma$ so that by Stokes' theorem
  \begin{equation*}
    \int_u d\alpha = \int_\gamma \alpha \;.
  \end{equation*}
  Since $P$ is weakly exact, the left integral had to be $0$, while
  the right integral has to be strictly positive, because
  $\alpha(\gamma') > 0$ everywhere.
  Thus it follows that $\gamma$ cannot be contractible in $V$.
  For the opposite direction, assume now that every smooth loop in $P$
  that is positively transverse to the foliation is non-trivial in
  $\pi_1(V)$.
  To show that $P$ is weakly exact, we have to prove that for any
  smooth map $u\colon(\DD^2,\p \DD^2) \to (V,P)$ the integral
  $\int_{\DD^2} u^*d\alpha = \int_{\p \DD^2}u^*\alpha$ is $0$.
  We show below that every loop~$\gamma$ in $P$ with
  $\int_\gamma \alpha > 0$ can be homotoped to one that is positively
  transverse to $\fF$.
  {Our starting assumption then implies that none of the
    loops~$\gamma \subset P$ with $\int_\gamma \alpha \ne 0$ can be
    contractible in $V$, and since the boundary of a disk~$u$} clearly
  is contractible, {we obtain} $\int_{\p \DD^2}u^*\alpha = 0$ as
  we wanted to show.
  It remains to prove that every smooth loop
  $\gamma\colon \SS^1 \to P$ satisfying $\int_\gamma \alpha > 0$ is
  isotopic to a smooth loop $\tilde\gamma\colon \SS^1 \to P$ that is
  everywhere positively transverse to $\fF$.
  By assumption, $C = \int_{\gamma} \alpha$ is positive, and we define
  $g(t) = \alpha( \gamma'(t) )$ so that
  $\int_0^{2\pi} g(t)\,dt = \int_{\gamma} \alpha = C$.
  Set
  \begin{equation*}
    f(t) = \frac{Ct}{2\pi} - \int_0^t g(s) \, ds \;,
  \end{equation*}
  and observe that $f(0) = 0 = f(2\pi)$ and that
  $f'(t) = C/(2\pi) - g(t)$.
  Choose any vector field~$Y$ on $P$ such that $\alpha(Y) = 1$, and
  let $\Phi_t(x) = \Phi_t^Y(x)$ denote its time-$t$ flow.
  Note that $\Phi_t$ preserves $\restricted{\alpha}{TP}$.
  For every $\tau \in [0,1]$, the map
  \begin{equation*}
    t \mapsto \Phi_{\tau f(t)}( \gamma(t) )
  \end{equation*}
  provides a smooth loop $\SS^1 \to P$, and for $\tau = 1$ we obtain
  \begin{equation*}
    \alpha\left( \tfrac{d}{dt} \bigl[ \Phi_{f(t)}(\gamma(t)) \bigr] \right)
    = \alpha \Bigl( f'(t)\,Y + \Phi_{f(t)*} \gamma'(t) \Bigr) 
    = f'(t) + g(t) = \tfrac{C}{2\pi} > 0 \;.
  \end{equation*}
  This thus constructs the desired isotopy.
\end{proof}

The link between weakly exact pre-Lagrangians and weakly exact
Lagrangians is established by the following lemma whose proof is an
easy exercise using the tautological property and Stokes' theorem (see
also
\cite{MassotNiederkrugerNonconnectedContactomorphisms}*{Lemma~2.2}).

\begin{lemma}\label{lemma: exact pre-Lagrangian and exact Lagrangian
    in symplectization}
  A pre-Lagrangian $P\subset (V,\xi)$ is weakly exact if and only if
  its Lagrangian lifts are weakly exact in the
  symplectization~$(SV, d\lcan)$.
\end{lemma}

\vspace{0.3cm}

Recall that a Lagrangian~$L$ is called \defin{displaceable} if there
is a compactly-supported Hamiltonian isotopy $\phi_t$ such that
$\phi_1(L) \cap L = \emptyset$.
{Accordingly}, a pre-Lagrangian~$P$ is called
\defin{displaceable} if there is a contact isotopy $\phi_t$ such that
$\phi_1(P) \cap P = \emptyset$.
The proof of Theorem~\ref{thm: main obstruction result} uses 
the following result by Gromov as an essential ingredient:

\begin{theorem}\label{thm: Gromov}
  Let $(W,d\lambda)$ be an exact symplectic manifold, convex at
  infinity.
  \begin{itemize}
  \item [(a)] There are no closed, weakly exact Lagrangians in
    $(W\times\CC, d\lambda\oplus dz\wedge d\bar z)$, see
    \cite{Gromov_HolCurves}*{Section~2.3 B$_3$}.
  \item [(b)] There are no closed, displaceable weakly exact Lagrangians in
    $(W,d\lambda)$, see \cite{Gromov_HolCurves}*{Section~2.3~B$_3'$}.
  \end{itemize}
\end{theorem}

With this we are ready to prove Theorem~\ref{thm: main obstruction
  result} which simply translates the statements above to certain
pre-Lagrangians to give obstructions to (subcritical) Weinstein
fillability.
Let $(V,\xi)$ be a regular level set~$f^{-1}(c)$ of a Weinstein
manifold~$(W,\omega, X, f)$ such that all critical values of $f$ are
smaller than $c$.
Using that the flow~$\Phi_t^X$ of the Liouville field is by assumption
complete, we construct a symplectic embedding
$j\colon SV \hookrightarrow W$ of the symplectization of $V$ in $W$.
The image~$j(SV)$ is dense in $W$, and its complement consists only of
the Lagrangian skeleton of $W$, that is, $W\setminus j(SV)$ is the
union of the stable manifolds of the critical points of $X$.
These are all of dimension~$n \le \frac{1}{2}\,\dim W$ and of
dimension~$n < \frac{1}{2}\,\dim W$ if $W$ is subcritical.
Also the image of a closed Lagrangian $L\subset SV$ is clearly a
closed Lagrangian $j(L)$ in $W$.
By Lemma~\ref{lemma: exact pre-Lagrangian and exact Lagrangian in
  symplectization}, a weakly exact pre-Lagrangian in $V$
gives rise to a weakly exact Lagrangian in the symplectization~$SV$.
In the context of Weinstein fillings, we have the following stronger
result (see also the proof of
\cite{BiranKhanevsky}*{Proposition~5.1}):

\begin{lemma}\label{lemma: exact Lagrangian in symplectization and in
    Stein}
  Let $(W,\omega, X, h)$ be the Weinstein filling of a contact
  manifold $(V,\xi)$, and let $P \subset (V,\xi)$ be a pre-Lagrangian
  with Lagrangian lift $L\subset (SV,d\lambda)$.
  Assume that either $\dim W\ge 6$, or that $W$ is subcritical and
  $\dim W = 4$.
  The pre-Lagrangian~$P$ is weakly exact if and only if
  $j(L)\subset W$ is weakly exact in $W$.
\end{lemma}
\begin{proof}
  Since every map $v\colon (\DD^2, \p \DD^2) \to (SV,L)$ can be viewed
  as a map into $W$, the weak exactness of $j(L)\subset W$ implies
  directly the one of $L \subset SV$.
  By Lemma~\ref{lemma: exact pre-Lagrangian and exact Lagrangian in
    symplectization} it then follows that $P$ is also weakly exact.
  Assume now that $L\subset SV$ is a weakly exact Lagrangian and let
  $u\colon (\DD^2,\p \DD^2) \to \bigl(W,j(L)\bigr)$ be a smooth map.
  Since the skeleton of $W$ is only $n$-dimensional, and since
  $\dim W = 2n \ge 6$, we see that $n+2 < 2n$ so that the image of $u$
  will generically not intersect the skeleton of $W$.
  After a homotopy, we may assume that the image of $u$ lies in the
  complement of the Lagrangian skeleton and thus in $j(SV)$, and we
  may apply the weak exactness assumption of $L$ in $SV$.
  If $\dim W=4$, but if $W$ is subcritical, then we arrive to the same
  conclusion because the skeleton of $W$ is only $1$-dimensional.
\end{proof}

\begin{theorem_obstruction_subcrit}
  A closed contact manifold containing a weakly exact
  pre-Lagrangian~$P$ is not subcritically Weinstein fillable.
  If the dimension of the contact manifold is at least $5$ and if $P$
  is displaceable then it follows that the contact manifold is not
  even Weinstein fillable.
\end{theorem_obstruction_subcrit}
\begin{proof}
  Combining Theorem~\ref{thm: Cieliebak} by Cieliebak with
  Theorem~\ref{thm: Gromov}.(a) by Gromov, we see that subcritical
  Weinstein manifolds do not contain any weakly exact Lagrangians.
  As a consequence of Lemma~\ref{lemma: exact Lagrangian in
    symplectization and in Stein}, it then follows that contact
  manifolds that are subcritically fillable may not contain weakly
  exact pre-Lagrangians. This proves the first statement of the
  theorem.
  According to
  \cite{MassotNiederkrugerNonconnectedContactomorphisms}*{Lemma~2.4} a
  contact isotopy that displaces a pre-Lagrangian lifts to a
  Hamiltonian isotopy with compact support in the symplectization that
  displaces a Lagrangian lift of the pre-Lagrangian.
  Using Theorem~\ref{thm: Gromov}.(b) combined with Lemma~\ref{lemma:
    exact Lagrangian in symplectization and in Stein} we then see that
  a Weinstein fillable contact manifold of dimension at least $5$ may
  not contain any displaceable weakly exact pre-Lagrangians.
\end{proof}

\begin{example}
  Let $(W,\omega)$ be a closed manifold with an integral symplectic
  form so that we can find a principal circle bundle
  $\pi\colon V \to W$ with Euler class~$[\omega]$.
  The pre-quantization~$(V,\alpha)$ is a contact manifold where we
  choose a contact form~$\alpha$ such that $d\alpha = \pi^*\omega$ and
  $\alpha(Z) = 1$ for $Z$ the infinitesimal generator of the circle
  action {(this implies that $\alpha$ is invariant under the
    circle action and $Z$ is its associated Reeb vector field)}.
  {Any Lagrangian~$L$ in $W$ is covered by a
    pre-Lagrangian~$P = \pi^{-1}(L)$ in $V$}, because
  {$\restricted{\alpha}{TP}$} is not singular and
  {$\restricted{d\alpha}{TP} = \pi^*(\restricted{\omega}{TL}) =
    0$}, see \cite{EliashbergHoferSalamon}.
  Furthermore {if $L$ is weakly exact so is} $P$, because if
  $u\colon (\DD^2,\p \DD^2) \to (V,P)$ is any smooth map, then we
  obtain with a simple calculation
  \begin{equation*}
    \int_{\DD^2} u^*d\alpha = \int_{\DD^2} u^*\pi^*\omega =
    \int_{\DD^2} (\pi\circ u)^*\omega = 0 \;,
  \end{equation*}
  using that $\pi\circ u$ is a smooth map in $W$ with boundary in $L$.
  A pre-quantization over a symplectic manifold containing a weakly exact
  Lagrangian is thus not subcritically Weinstein fillable.
\end{example}

\begin{lemma}\label{lemma: legendrian in page weakly exact in bourgeois}
  Let $(V,\xi)$ be a contact manifold with compatible open book
  $(K,\vartheta)$, and equip $V\times \TT^2$ with the Bourgeois
  structure corresponding to this open book.
  {We have the following two constructions of pre-Lagrangians:}
  \begin{itemize}
  \item [(a)] If there is a closed Legendrian $L\subset V$ contained
    in one page of $(K, \vartheta)$ then
    $L\times \TT^2\subset V\times \TT^2$ is a weakly exact
    pre-Lagrangian.
  \item [(b)] {Any closed pre-Lagrangian~$P\subset K$ in the
      binding~$(K,\xi\cap TK)$ yields a pre-Lagrangian~$P\times \TT^2$
      in the Bourgeois manifold.
      This pre-Lagrangian is weakly exact if and only if every loop in
      $P$ that is positively transverse to the characteristic
      foliation is non-contractible in $V$.}
  \end{itemize}
\end{lemma}

Part~(b) also applies directly to more general Bourgeois-Gironella
structures on $V\times \Sigma$.

\begin{proof}
  {\textbf{(a)}} We will first show that $L\times \TT^2$ is a
  pre-Lagrangian.
  Let $\alpha_V$ be a contact form for $\xi$ that is supported by the
  open book $(K, \vartheta)$.
  The Bourgeois structure on $V\times \TT^2$ is given as the kernel of
  the form
  $\alpha = \alpha_V + f_x \, d\varphi_1 - f_y\, d\varphi_2$
  where $(\alpha_V, f_x + i f_y)$ is a representation of
  $(K, \vartheta)$ and $\xi$.
  If $L$ is Legendrian, then $\restricted{\alpha_V}{TL} = 0$.
  Since $L$ is contained in the interior of one of the pages, either
  $f_x$ or $f_y$ do not vanish anywhere on $L$.
  Suppose it is $f_x$, then we extend $\restricted{f_x}{L}$ to a
  nowhere vanishing function~$\hat f_x$ on $V \times \TT^2$ that we
  use to rescale $\alpha$.
  For this new contact form, we have
  $\restricted{\alpha}{T(L\times\TT^2)} = d\varphi_1 - c\, d\varphi_2$
  where $c$ is a \emph{constant}
  $\frac{f_y}{f_x} = \tan \restricted{\vartheta}{L}$.
  This implies that $L\times \TT^2$ is pre-Lagrangian.
  To see that $L\times \TT^2$ is weakly exact choose any loop~$\gamma$
  that is positively transverse to the foliation of $L\times \TT^2$
  given by $\ker\alpha$.
  According to Lemma~\ref{lemma: exact pre-Lagrangians and positive
    loops}, $L\times \TT^2$ is weakly exact if $\gamma$ is not
  contractible in $V\times \TT^2$, i.e. non-trivial in
  $\pi_1(V\times \TT^2) = \pi_1(V)\times\pi_1(\TT^2)$.
  Since the characteristic foliation on $L\times \TT^2$ is the lift of
  the linear foliation on $\TT^2$, it follows that $\gamma$ projects
  to a non-trivial loop in $\pi_1(\TT^2)$.
\textbf{(b)} Let $P$ be a pre-Lagrangian in the binding~$K$.
   Notice that both functions $f_x$ and $f_y$ vanish along $K \times \TT^2$ so
   in particular along $P \times \TT^2$. Notice also that $P \times \TT^2$ is
   of the correct dimension. It follows therefore that $P \times
   \TT^2$ is pre-Lagrangian.
   The statement about weak exactness follows immediately from 
  Lemma~\ref{lemma: exact pre-Lagrangians and positive loops}, and that $\pi_1(V\times \TT^2)$ decomposes as a product.

  \end{proof}

Corollary~\ref{cor: page contains Legendrian then Bourgeois not
  subcritical} from the introduction now follows immediately from
Theorem~\ref{thm: main obstruction result} and Lemma~\ref{lemma:
  legendrian in page weakly exact in bourgeois}(a).
As an application of Theorem~\ref{thm: main obstruction result} we are
able to show that even though some Bourgeois contact structures are
subcritically fillable, most are not.
In particular, we see that changing the open book for a given contact
structure may destroy the subcritical fillability of the resulting
Bourgeois structure.

\appendix

\section{Contact open books and ideal Liouville
  domains}\label{appendix}

The aim of this appendix is to give a short overview on \emph{ideal
  Liouville domains} introduced by Giroux \cite{GirouxIdealDomains}
and illustrate their use by working out a few classical examples of
contact open books.
Even though the relation between contact open books and abstract open
books is by now well-known and has been discussed in several sources
(e.g. \cites{Giroux_ICM, EtnyreLectOpenBooks, KoertLectOpenBooks}), to
the best of our knowledge there is no unified treatment in the
literature that does not have some missing details.
One of the key sticky points has to do with smoothing of corners and
modifying monodromy maps correctly near the boundary of the page.
These difficulties are encapsulated in the ideal Liouville domain
machinery, and are dealt with by Giroux in his abstract framework
\cite{GirouxIdealDomains}.
An abstract Liouville open book consists of an ideal Liouville domain
together with a monodromy map (see below).
The main result reads as follows.

\begin{theorem}[Giroux]\label{theorem: equivalence abstract and
    embedded open books}
  There is a natural bijection between homotopy classes of contact
  structures supported by an open book and homotopy classes of
  abstract Liouville open books.
\end{theorem}

{In the next three sections,} we explain the formalism
introduced by Giroux and illustrate it by applying it to the two most
elementary open books on the standard contact sphere.
We also verify that an open book with trivial monodromy is explicitly
fillable by the Weinstein manifold obtained by stabilizing the page
(in the sense of Weinstein domains, see Section~\ref{sec: explicit
  fillings}).
%
%\marginpar{\Klaus{I added ``(in the sense of Section~\ref{sec:
%      explicit fillings})'' to make it clear that the stabilization is
%    a Weinstein type stabilization, and not a stabilization of an open
%    book, but I do not know if this solution is any clearer?}}
%%% In \ref{appendix: branched cover}, we look at branched covers
%%% along the binding of an open book.
%

\subsection{Contact open books}\label{appendix: contact open book}

\begin{definition} \label{def: smooth open book}
  Let $V$ be a closed manifold.
  An \defin{open book} on $V$ is a pair~$(K, \vartheta)$ where:
  \begin{itemize}
  \item $K \subset V$ is a non-empty codimension-$2$ submanifold with
    trivial normal bundle;
  \item $\vartheta\colon V \setminus K \to \SS^1$ is a fibration that
    agrees in a tubular neighborhood $K \times \DD^2$ of
    $K = K \times \{0\}$ with a normal angular coordinate.
  \end{itemize}
  We call $K$ the \defin{binding} of the open book, and we call the
  closure $F_\varphi := \vartheta^{-1}\bigl(e^{i\varphi}\bigr) \cup K$
  of every fiber a \defin{page} of the open book.
\end{definition}

Note that the pages are smooth {compact} submanifolds with
boundary~$K$.

\begin{remark}\label{rmk: top open book specified by complex function}
  It is easy to see that an open book can equivalently be specified by
  a smooth function $h\colon V \to \CC$ for which $0$ is a regular
  value such that $K_h \coloneqq h^{-1}(0)$ is not empty, and such
  that
  \begin{equation*}
    \vartheta_h\colon V \setminus K_h \to \SS^1, \,
    p\mapsto \frac{h(p)}{\abs{h(p)}}
  \end{equation*}
  is a submersion.
  The set of smooth functions defining a given open book is a
  non-empty convex subset of $C^\infty(V,\CC)$.
\end{remark}

{ If $V$ is an oriented manifold, the coorientations specified
  by $\vartheta$ orient both the pages and the binding.
From a practical viewpoint it is helpful to formulate these
orientations using volume forms.
\begin{itemize}
\item A vector~$R\in T_p V$ at a point $p\in V\setminus K$ is
  positively transverse to a page if and only if $d\vartheta(R) > 0$.
  Given a positive volume form~$\Omega_V$ on $V$, it follows that
  $\iota_R \Omega_V$ determines the positive orientation for the page.
  A volume form~$\Omega_F$ on a page~$F_\varphi$ is thus positive if
  and only if $d\vartheta \wedge \Omega_F$ is positive on
  $\restricted{TV}{\Int F_\varphi}$.
\item Identify the neighborhood of $K$ with $K\times \DD^2$ such that
  the angular coordinate~$\varphi$ agrees with $\vartheta$ and such
  that the disk has the canonical orientation with
  coordinates~$(x,y)\in \DD^2$.
  Then it follows that for a positive volume form~$\Omega_V$ on $V$
  the restriction of $\iota_{\partial_y} \iota_{\partial_x}\Omega_V$
  is a positive volume form on the binding.
  Conversely, a volume form $\Omega_K$ on $K$ is positive if and only
  if
  $dx\wedge dy \wedge \Omega_K = r\,dr\wedge d\varphi\wedge \Omega_K$
  is a positive volume element on $\restricted{TV}{K}$.
\end{itemize}}
Note that with these orientations the binding is oriented as the
boundary of the pages.

\begin{definition}[{\cite{Giroux_ICM}}] 
    \label{def: contact open book classical}
  Let $(V,\xi)$ be a closed contact manifold.
  We say that $\xi$ is \defin{supported} by an open book
  decomposition~$(K, \vartheta)$ of $V$, if $\xi$ admits a contact
  form~$\alpha$ such that
  \begin{enumerate}[(i)]
  \item The binding~$K$ is a contact submanifold with
    {\emph{positive}} contact
    form~$\alpha_K \coloneqq \restricted{\alpha}{TK}$.
  \item The restriction of $d\alpha$ to the interior $\Int F_\varphi$
    of every page is a \emph{positive} symplectic form.
  \end{enumerate}
  {In both cases, ``positive'' refers to the orientation induced
    on $K$ and the pages by the open book decomposition.}
  We call a contact form~$\alpha$ as above, \defin{adapted} to the
  open book, and we call $(\xi, K,\vartheta)$ a \defin{contact open
    book decomposition}.
\end{definition}

The following remark is in a way an extension of Remark~\ref{rmk: top
  open book specified by complex function} to the contact category.

{
\begin{remark}
  Let $V$ be a closed manifold, let $\alpha$ be a contact form with
  Reeb field~$R_\alpha$, and let $h = h_x + i h_y\colon V\to \CC$ be a
  smooth function.
  To show that $K_h \coloneqq h^{-1}(0)$ and $\vartheta_h = h/\abs{h}$
  define an open book~$(K_h, \vartheta_h)$ and that $\alpha$ is
  adapted to it, it suffices to verify that
  \begin{itemize}
  \item[(i)] $K_h$ is non-empty and
    $\alpha\wedge (d\alpha)^{n-1}\wedge dh_x\wedge dh_y$ is positive
    along $K_h$;
  \item[(ii)]
    $\frac{i}{2}\, \bigl(h\,d\bar h(R_\alpha) - \bar
    h\,dh(R_\alpha)\bigr) = h_x\,dh_y(R_\alpha) - h_y\, dh_x(R_\alpha)
    > 0$ on all of $V\setminus K_h$.
  \end{itemize}
  (Here, as above, $\bar h = h_x - i h_y$ denotes the complex conjugate.)
\end{remark}
\begin{proof}
  By condition~(i), $0$ is a regular value of $h$.
  Recall that
  $d\vartheta_h = \frac{i}{2\abs{h}^2}\, \bigl(h\,d\bar h - \bar
  h\,dh\bigr) = \frac{1}{\abs{h}^2}\, \bigl(h_x\, dh_y - h_y\,
  dh_x\bigr)$,
  thus condition~(ii) simply implies that
  $d\vartheta_h(R_\alpha) > 0$, and it follows that $\vartheta_h$ is a
  submersion.
  Thus $h$ defines an open book by Remark~\ref{rmk: top open book
    specified by complex function}.
  Let us now show that $\alpha$ is adapted to the open
  book~$(K_h,\vartheta_h)$.
  Since $\alpha\wedge (d\alpha)^n$ is a volume form,
  $\iota_{R_\alpha} \alpha\wedge (d\alpha)^n = (d\alpha)^n$ cannot be
  degenerate on any hyperplane transverse to $R_\alpha$.
  In particular, because $d\vartheta_h(R_\alpha) > 0$, the Reeb field
  is positively transverse to the interior of the pages, and $d\alpha$
  restricts to a positive symplectic form on them.
  {Because $TK_h$} lies in the kernel of the $2$-form
  $d\bigl(\abs{h}^2\, d\vartheta_h\bigr) = 2\,dh_x\wedge dh_y$,
  condition~$(i)$ implies then that $\alpha$ restricts to a contact
  form on $K_h$.
  Furthermore $d\bigl(\abs{h}^2\, d\vartheta_h\bigr)$ defines the
  positive coorientation for the binding, thus
  $\restricted{\alpha}{TK_h}$ is by~(i) positive.
\end{proof}
}

\vspace{0.8cm}

We now describe two elementary examples of contact open book
decompositions of the standard sphere that we will study in detail in
the next two sections of this appendix using the language of
\cite{GirouxIdealDomains}.

\begin{example}\label{example: contact open book decomposition for
    sphere}
  We assume that the unit sphere $\SS^{2n-1}\subset \CC^n$ is equipped
  with the standard contact structure~$\xi_0$, which is the hyperplane field
  of complex tangencies. Equivalently, this is given as the kernel of the
 $1$-form
  \begin{equation*}
    \alpha_0 = \frac{1}{2}\, \sum_{j=1}^n \bigl(x_j\,dy_j - y_j\,dx_j\bigr) \;,
  \end{equation*}
  where we write the coordinates of $\CC^n$ as
  $\z = (z_1,\dotsc,z_n) = (x_1+iy_1,\dotsc,x_n+iy_n)$.
  Every holomorphic function~$g\colon \CC^n \to \CC$ with an isolated
  singularity at the origin induces a contact open book decomposition
  of the standard sphere (after possibly shrinking the radius of the
  sphere, see \cite{MilnorComplexHypersurfaces} for the
  {topological} and \cite{MilnorOpenBooks} for the contact
  case).
  For the concrete applications we have in mind here, we will not
  appeal to this general result, and instead study the following two
  very explicit situations.
  \textbf{(a)} Let $g_1(z_1,\dotsc,z_n) = z_1$, then the binding is
  the submanifold~$K=\{z_1=0\}$ and the fibration is
  $\vartheta\colon (z_1,\dotsc,z_n)\mapsto z_1/\abs{z_1}$.
  The binding~$K$ is just the standard contact sphere and the Reeb
  vector field $R_0 = (iz_1, \dotsc, iz_n)$ for $\alpha_0$ generates
  the Hopf fibration which is transverse to the (the interior) of
  every page~$F_\varphi = \{\arg z_1 = \varphi\}$ so that $d\alpha_0$
  will restrict to a symplectic form defining the correct orientation
  on every page.
  {Furthermore, since the binding~$K$ is connected, it follows from
      Stokes' Theorem that 
    $\alpha$ induces the boundary orientation of the pages on $K$.}
  It follows that $(K,\vartheta)$ is a contact open book
  decomposition.
  \textbf{(b)} Let us now study the case of
  $g_2(z_1,\dotsc,z_n) = z_1^2+ \dotsm + z_n^2$.
  The complex hypersurface~$V_{g_2} = g_2^{-1}(0)$ is everywhere
  smooth except at the origin and since it is invariant under linear
  scaling
  $\lambda\cdot (z_1,\dotsc,z_n) = (\lambda\cdot z_1,\dotsc,
  \lambda\cdot z_n)$
  with $\lambda \in \RR_+$, it follows that $V_{g_2}$ is transverse to
  $\SS^{2n-1}$ so that the binding~$K = g_2^{-1}(0) \cap \SS^{2n-1}$
  is a smooth codimension~$2$ submanifold of the standard sphere.
  To check the contact condition note that the restriction of a
  plurisubharmonic function to a complex submanifold preserves this
  property.
  In our case, the restriction of $\z \mapsto \norm{\z}^2$ to
  $V_{g_2}$ is such a function, and $K$ is one of its regular level
  set, so that $K$ is a contact submanifold.
  For the pages, notice that the Reeb field~$R_0$ increases the
  argument of $g_2$ everywhere where $g_2$ does not vanish.
  This implies that $R_0$ is positively transverse to the pages, and
  in particular $d\alpha_0$ defines a symplectic structure on them.
\end{example}

It is well-known from the ``classical'' treatment that the page in the
first example is a ball with the standard symplectic structure and
that its monodromy is the identity.
In the second example, the page is the cotangent bundle of the sphere
and the monodromy is a generalized Dehn twist.
In Examples~\ref{example: ideal Liouville domain} and \ref{example:
  from open book to abstract open book} below we will work out the
abstract open books in these two cases using the formalism of ideal
Liouville domains.

\subsection{Ideal Liouville domains}\label{appendix: ideal liouville}

As we already mentioned above, the pages of an abstract open book will
be described by an ideal Liouville domain.

\begin{definition} 
    \label{def:ideal Liouville domain}
  Let $F$ be a compact manifold with boundary~$K := \p F$, and let
  $\omega$ be an exact symplectic form on the \emph{interior} $\Int F
  = F \setminus K$.
  The pair $(F, \omega)$ is an \defin{ideal Liouville domain} if there
  exists a primitive $\lambda \in \Omega^{1}(\Int F)$ for $\omega$
  such that:
  For any smooth function $u\colon F \to [0, \infty)$ with regular
  level set $K = u^{-1}(0)$, the $1$-form~$u\lambda$ extends to a
  smooth $1$-form $\lambda_u$ on all of $F$ whose restriction to $K$
  is a (positive) contact form.
  Every such primitive~$\lambda$ is called a \defin{Liouville form} of
  $(F, \omega)$.
\end{definition}

The intuitive picture of an ideal Liouville domain is that of a
classical Liouville domain that has been completed by attaching a
cylindrical end and has then been compactified by fixing a certain
asymptotic information at ``infinity'' that is captured in the
boundary of the ideal Liouville domain.
Below we give a formal description of this completion process.
For the many properties shared by these objects, we refer to
\cite{GirouxIdealDomains}.
In particular we point out that the contact structure induced on the
boundary is, {as observed by Courte}, already determined by
$(F,\omega)$ itself and does not depend on the auxiliary Liouville
form chosen \cite{GirouxIdealDomains}*{Proposition~2}.
We denote the space of all diffeomorphisms of $F$ that keep the
boundary pointwise fixed and that preserve $\omega$ on the interior by
$\Diff_\p(F,\omega)$.

\begin{definition} \label{def: abstract Liouville open book}
  An \defin{abstract Liouville open book} consists of an ideal
  Liouville domain~$(F,\omega)$ and a diffeomorphism
  $\phi \in \Diff_\p(F,\omega)$.
\end{definition}

\vspace{0.8cm}

We will now describe the \defin{completion of classical Liouville
  domains} allowing us to do the transition from classical to ideal
Liouville domains.
Recall that a \defin{``classical'' Liouville domain}~$(F,\lambda_c)$
is a compact manifold with boundary~$K$ such that
\begin{itemize}
\item $\omega_c := d\lambda_c$ is a symplectic form;
\item the Liouville vector field~$X_\lambda$ defined by the equation
  $\iota_{X_\lambda} \omega_c = \lambda_c$ points along $K$
  transversely out of the domain~$F$.
\end{itemize}
In particular it follows that $\lambda_c$ restricts on $K$ (oriented
as the boundary of $(F, \omega_c)$) to a positive contact form.
Following \cite{GirouxIdealDomains}*{Example~9}, we will convert
$(F,\lambda_c)$ into an ideal Liouville domain $(F,\omega)$, keeping
the smooth manifold~$F$ unchanged, but modifying $d\lambda_c$ to a new
symplectic form~$\omega$ on $\Int F$ (that will be related to but
different from $\omega_c$!).

\begin{lemma}\label{lemma: function u for completion}
  The space of all functions $u\colon F\to [0,\infty)$ satisfying
  \begin{itemize}
  \item $u^{-1}(0) = K$ is a regular level set;
  \item $X_\lambda(\ln u) < 1$ on $\Int F$ (or equivalently
    $du(X_\lambda) < u$ on all of $F$)
  \end{itemize}
  is convex and non-empty.
\end{lemma}
\begin{proof}
  Convexity is a basic calculation; for the existence use a collar
  neighborhood $(-\epsilon,0]\times K$ with coordinates $(t,x)$
  defined by the flow of $X_\lambda$, and let $u(t,x)$ be a function
  that agrees with $-t$ close to $t=0$, and flattens out to be
  constant on a slightly larger neighborhood of $K$.
\end{proof}

With a function~$u$ as in the previous lemma, we claim that
$\omega := d\bigl(\lambda_c/u\bigr)$ is an ideal Liouville structure
on $F$.
  Firstly, $\omega$ is symplectic on $\Int F$:
At points where $\lambda_c = 0$ we have
$\omega = \frac{1}{u}\, \omega_c$; to check the non-degeneracy of
$\omega$ at the remaining points note first that $\lambda_c$ vanishes
if and only if $X_\lambda$ does, then compute
\begin{equation*}
  \omega^n = \bigl(\frac{1}{u}\, d\lambda_c -
  \frac{1}{u^2}\,du\wedge \lambda_c\bigr)^n = \frac{1}{u^n}\,\bigl(
  \omega_c^n - n\,d(\ln u)\wedge \lambda_c\wedge
  \omega_c^{n-1}\bigr) \;.
\end{equation*}
Plugging $X_\lambda$ into $\omega^n$ and using that
$\iota_{X_\lambda}\omega_c = \lambda_c$ and
$\iota_{X_\lambda}\lambda_c = 0$, we see that
\begin{equation*}
  \iota_{X_\lambda}\omega^n = \frac{n}{u^n}\, \bigl(1 -  X_\lambda(\ln u)\, \bigr)
  \, \lambda_c\wedge \omega_c^{n-1} =  \frac{1}{u^n}\, \bigl(1 -  X_\lambda(\ln u)\, \bigr)  
  \, \iota_{X_\lambda} \omega_c^n \;,
\end{equation*}
which is non-degenerate.
This implies now that $(F,\omega)$ is an ideal Liouville domain,
because $\lambda := \frac{1}{u}\, \lambda_c$ is a primitive of
$\omega$ for which $u\, \lambda$ clearly restricts to a contact form
on $K$.
The contact structure on the boundary of $(F,\omega)$ is equal to the
initial contact structure.
As explained in \cite{GirouxIdealDomains}*{Example~9}, one can
equivalently obtain $(F,\omega)$ by attaching an infinite cylindrical
end to the boundary and then compactify this.
Also note that by the convexity of the admissible choices for $u$ the
completion is unique up to isotopy.
%

%\vspace{0.3cm}
\

The completion of a classical Liouville domain 
to an ideal Liouville domain is particularly
straightforward when applied to a Weinstein
domain, or equivalently for our purposes, a Weinstein 
manifold~$(W,\omega, X, f)$ of finite--type.
In this case, choose a regular value~$C$ such that
$F:=\bigl\{f\le C\bigr\}$ is a non-empty domain with smooth boundary.
In particular $(F,\lambda)$ with the Liouville form
$\lambda :=\iota_X \omega$ is a classical Liouville domain, and the
function $u := C - f$ is non-negative on $F$, has the
boundary~$K=\p F$ as regular level set, and since $f$ is a Lyapunov
function for $X$, we check that $X(\ln u) = -\frac{1}{u}\, X(f) \le 0$
is always smaller than $1$.
The interior of the ideal Liouville domain~$F$ is symplectomorphic to
cutting off the part $\bigl\{f > C\bigr\}$ from $W$ and replacing it
by the cylindrical end of the level set $\{f=C\}$. By choosing $C$
sufficiently large, this then recovers the Weinstein manifold 
of finite-type $W$.

\

We will now illustrate the notion of an ideal Liouville domain with
two basic examples obtained via this completion procedure.
As we will see in the next section, these two examples correspond to
the pages of the open books from
Example~\ref{example: contact open book decomposition for sphere}.

\begin{example}\label{example: ideal Liouville domain}
  \textbf{(a)} Let $\overline{\DD}^{2n}$ be the closed unit disk in
  $(\CC^n, \omega_0 = d\lambda_0)$ with coordinates
  $\z = \x+i\y = (x_1+iy_1,\dotsc,x_n+iy_n)$ and let
  $\lambda_0 = \frac{1}{2}\, \sum_{j=1}^n (x_j\,dy_j - y_j\,dx_j)$ be
  the standard Liouville form.
  We could of course use the fact that $\CC^n$ is a Weinstein manifold
  with Liouville vector field
  $X_\lambda = \frac{1}{2}\, \sum_{j=1}^n
  \bigl(x_j\,\frac{\partial}{\partial x_j} +
  y_j\,\frac{\partial}{\partial {y_j}}\bigr)$
  and Lyapunov function $f(\z) = \norm{\z}^2$ to apply the remark we
  just made.
  Instead we will perform the completion procedure using the function
  \begin{equation*}
    u\colon \overline{\DD}^{2n} \to [0,\infty),\, \z \mapsto
    1-\norm{\z}^4 \;.
  \end{equation*}
  The reason we make this unexpected choice for $u$ is to
  recover the page of the abstract open book in
  Example~\ref{example: from open book to abstract open book} below.
  Recall that up to symplectomorphism, the ideal Liouville domain does
  not depend on the particular choice of the function satisfying the
  properties of Lemma~\ref{lemma: function u for completion}.
  Note first that the boundary of the closed disk is a regular level
  set of $u$ since $u$ factors as
  $u(\z) = (1-\norm{\z})\, (1+\norm{\z})\,(1+\norm{\z}^2)$.
  Furthermore, $X_\lambda(\norm{\z}^4) \ge 0$ so that
  $X_\lambda\bigl(\ln (1-\norm{\z}^4)\bigr) \le 0$.
  Setting $\lambda = \frac{1}{1-\norm{\z}^4}\, \lambda_0$, we obtain
  that $\bigl(\overline{\DD}^{2n}, d\lambda)$ is the desired
  completion of $(\overline{\DD}^{2n}, \lambda_0)$ to an ideal
  Liouville domain.
  The interior of the ideal Liouville domain
  $\bigl(\overline{\DD}^{2n}, d\lambda)$ is symplectomorphic to
  $(\CC^n, \omega_0 = d\lambda_0)$:
  Simply use the diffeomorphism
  $\DD^{2n} \to \CC^n, \, \z \mapsto \frac{1}{\sqrt{1-\norm{\z}^4}}\,
  \z$ to pull-back $\lambda_0$.
  \textbf{(b)} Let us now see how to associate an ideal Liouville
  domain to a unit cotangent bundle.
  For this, let $(L, g)$ be a closed Riemannian manifold, and let
  $\lcan$ be the canonical $1$-form on $T^*L$.
  It is well-known that $(T^*L, d\lcan, X_\lambda, f)$ with
  $X_\lambda = \bfp\cdot \p_\bfp$ and $f(\bfq,\bfp) = \norm{\bfp}^2$
  is a Weinstein manifold.
  As described above we can apply the completion using the function
  $u=1-f$ so that the (closed) unit disk bundle
  $\overline{\DD}(T^*L) = \bigl\{ (\bfq,\bfp) \in T^*L \bigm|\,
  \norm{\bfp} \le 1\bigr\}$
  is an ideal Liouville domain with the symplectic structure given by
  $d\lambda$ where we have set
  $\lambda = \frac{1}{1-\norm{\bfp}^2}\, \lcan$.
  In this case, we can identify the interior of
  $\bigl(\overline{\DD}(T^*L), d\lambda\bigr)$ with $(T^*L, d\lcan)$
  using the map
  $\DD(T^*L) \to T^*L, \, (\bfq, \bfp) \mapsto \bigl(\bfq,
  \bfp/(1-\norm{\bfp}^2)\bigr)$.

\end{example}

\subsection{From contact open book decompositions to abstract
  Liouville open books and back}\label{appendix: open book to
  abstract}

The link between abstract Liouville open books and contact open books
is established by the following intermediate object.

\begin{definition} \label{definition:Liouville open book}
  A \defin{Liouville open book} $(K, \vartheta, \omega_t)$ on a closed
  manifold~$V$ is an open book $(K,\vartheta)$, each of whose
  pages~$F_t$ is equipped with a (positive) ideal Liouville
  structure~$\omega_t \in \Omega^2(\Int F_t)$.
  To guarantee a certain compatibility between the $\omega_t$, we
  require that there is
  \begin{itemize}
  \item a global smooth $1$-form~$\beta$ on $V$ {called} a
    \defin{binding form} and
  \item a function $f\colon V \to \CC$ defining the open book (as in
    Remark~\ref{rmk: top open book specified by complex function})
  \end{itemize}
  such that $\beta/\abs{f}$ restricts on the interior of each page
  $\Int F_t = F_t \setminus K$ to a Liouville form of $\omega_t$.

  We say that the pair~$(\beta, f)$ is a \defin{representation} of the
  Liouville open book.
\end{definition}

Note in particular that a binding form induces a positive contact form
on the binding, since $\beta$ restricts on the boundary of each page
to a contact form.
We often make use of the following technical lemma.

\begin{lemma}\label{lemma: volume form from open book}
  Let $(K, \vartheta, \omega_t)$ be a Liouville open book on a
  manifold~$V$.
  Choose a representation~$(\beta,f)$ such that
  $\lambda := \beta/\abs{f}$ restricts on the interior of each page
  $\Int F_t = F_t \setminus K$ to a Liouville form of
  $\restricted{\omega_t}{\Int F_t}$.
  Then it follows that
  \begin{equation*}
    \Omega_V := \abs{f}^{n+2}\, d\vartheta\wedge  (d\lambda)^n
  \end{equation*}
  extends to a well-defined volume form on all of $V$.
  Furthermore, writing $f = \rho\, \e^{i \vartheta}$, we have
  \begin{equation}
    \Omega_V = n\, \rho\, d\rho \wedge d\vartheta\wedge \beta \wedge
    (d\beta)^{n-1} + \rho^2\, d\vartheta\wedge(d\beta)^n \;.
    \label{eq: volume form from open book}
  \end{equation}
\end{lemma}
\begin{proof}
  It is clear that $\Omega_V$ is a volume form on $V\setminus K$, so
  it only remains to analyze its behavior along the binding.
  Writing $f$ in polar coordinates, and replacing $\lambda$ by
  $\beta/\rho$, {we obtain \eqref{eq: volume form from
      open book} whose right--hand side} is defined on all of $V$.
  {Its} second term vanishes along the binding while the first
  one is positive, since the binding itself is a positive contact
  submanifold of $(V,\xi)$.
  This proves that $\Omega_V$ is a volume form.
\end{proof}

\vspace{0.8cm}

Table~\ref{table: relationship between types of open book} summarizes
the three notions introduced so far and their relationships:

\begin{table}[h]
\begin{center}
    \begin{tabular}{ m{0.35\textwidth} c m{0.28\textwidth}  c m{0.3\textwidth}}
      \textbf{abstract Liouville open book} \newline
      ideal Liouville domain~$(F,\omega)$ \newline
      diffeomorphism~$\phi \in \Diff_\p(F,\omega)$
      & $\leftrightarrow$
      & \textbf{Liouville open book} \newline
        open book~$(K,\vartheta)$ on $V$ \newline
        ideal Liouville structure~$\omega_t$ on each page~$F_t$
      & $\leftrightarrow$
      & \textbf{contact open book} \newline
        open book~$(K,\vartheta)$ on $V$ \newline
        contact structure~$\xi$ on $V$
    \end{tabular}
  \end{center}
  \caption{The different types of open books and their relationship.}
  \label{table: relationship between types of open book}
\end{table}

\vspace{0.8cm}

The connection between contact open books and Liouville open books is
the following:
A contact structure is said to be \defin{symplectically supported} by
a Liouville open book if it admits a contact form that is a binding
form.

\begin{proposition}[\cite{GirouxIdealDomains}*{Proposition~18}]\label{prop:contact_open_book_is_Liouville_ob}
  If $(K, \vartheta)$ is a contact open book on $V$ supporting the
  contact structure~$\xi$, and $f \colon V \to \CC$ is any defining
  function, then there exists a contact form~$\alpha$ such that
  $d(\alpha/\abs{f})$ restricts to each page as an ideal Liouville
  structure.
  Furthermore, for fixed $f$, the set of such forms~$\alpha$
  {is} a non-empty convex cone.  \qedhere
\end{proposition}

In other words, for each defining function $f$, there is a contact
form~$\alpha$ such that $(\alpha, f)$ is the representation of a
Liouville open book on $V$.
(Notice also that the space of defining functions for a given open
book is also convex and non-empty, so the space of pairs is
contractible.)
We also have a converse:

\begin{lemma}[\cite{GirouxIdealDomains}*{page
    19}]\label{lem:Liouville_OB_is_contact_OB}
  If the contact structure $\xi$ on $V$ is symplectically supported by
  a Liouville open book, then $\xi$ is supported (in the sense of
  Definition~\ref{def: contact open book classical}) by the underlying
  smooth open book.
\end{lemma}

These two facts justify our earlier definition that a pair
$(\alpha, f)$ is a representation of a contact open book decomposition
when {$\alpha$ is a contact form and $(\alpha, f)$} is a
representation of a Liouville open book with defining function~$f$.
Additionally, from \cite{GirouxIdealDomains}*{Proposition~21}, the
symplectically supported contact structures form a non-empty and
weakly contractible subset in the space of all hyperplane fields.

\vspace{0.3cm}

To describe now the connection between Liouville open books and
abstract Liouville open books, we need the following notion.

\begin{definition}
  Let $(K, \vartheta, \omega_t)$ be a Liouville open book on $V$.
  A smooth vector field~$X$ on $V$ is called a \defin{spinning vector
    field}, if it satisfies the following properties:
  \begin{itemize}
  \item $X$ vanishes along the binding~$K$, and $d\vartheta(X) = 2\pi$
    on $V\setminus K$,
  \item the flow of $X$ preserves the ideal Liouville structure on
    every page.
  \end{itemize}
  The \defin{monodromy} of the Liouville open book is a diffeomorphism
  $\phi\colon F_0 \to F_0$ obtained by restricting the time-$1$ flow
  of a spinning vector field~$X$ to the page~$F_0$.
  Clearly, $\phi$ is an element of $\Diff_\p(F_0,\omega_0)$.
\end{definition}

\begin{lemma}\label{lemma: exists vector field to recover monodromy}
  Let $(K, \vartheta, \omega_t)$ be a Liouville open book with a
  representation~$(\beta, f)$ so that
  $\lambda = \beta/\abs{f}$ restricts to a Liouville form on
  the interior of each page.
  There exists a unique vector field~$Y$ satisfying the two equations
  \begin{equation*}
    d\vartheta(Y) = 2\pi\quad \text{ and }
    \iota_Y d\lambda = 0 \;.
  \end{equation*}
  This field is a spinning vector field of the Liouville open book.
\end{lemma}

  \begin{proof}
    It is clear that $Y$ is defined on $V\setminus K$ and that its
    flow preserves the Liouville structure on the pages, but the
    properties of $Y$ along the binding are less obvious.
    Let us plug $Y$ into the volume form~$\Omega_V$ from
    Lemma~\ref{lemma: volume form from open book}
    \begin{equation*}
      \iota_Y\Omega_V = 2\pi\, \abs{f}^{n+2}\, (d\lambda)^n
      = 2\pi\,\abs{f}^2\,(d\beta)^n
      - \pi n\, d(\abs{f}^2)\wedge \beta\wedge (d\beta)^{n-1} \;.
    \end{equation*}
    Since the righthand side is defined on all of $V$, and since
    $\Omega_V$ is a volume form, we have found a defining equation for
    $Y$ that shows that $Y$ is everywhere smooth and vanishes along
    $K$.
  \end{proof}

\vspace{0.3cm}
One can easily obtain an abstract open book from a proper Liouville
open book by keeping only one of its pages and choosing the monodromy
with respect to any spinning vector field.
All pages are isomorphic, and spinning vector fields form a convex
subset so that all choices will lead to homotopic abstract Liouville
open books.
Starting from an abstract Liouville open book~$(F,\omega)$ with
diffeomorphism~$\phi \in \Diff_\p(F,\omega)$, Giroux constructs first
a mapping torus and then blows down its boundary to obtain the
binding, producing this way a Liouville open book, and thus the
desired bijection.

\begin{example}\label{example: from open book to abstract open book}
  Let us now come back to the contact open book decompositions
  introduced in Example~\ref{example: contact open book decomposition
    for sphere} and illustrate how to recover their abstract Liouville
  open books by applying the formalism explained above.
  \textbf{(a)} Recall that the function $g_1(z_1,\dotsc,z_n) = z_1$
  determines a contact open book~$(K, \vartheta)$ on the standard
  sphere.
  We will see that the page of the corresponding abstract open book is
  the ideal Liouville domain given by Example~\ref{example: ideal
    Liouville domain}.(a) and that the monodromy is the identity.
  This is then a special case of the more general \ref{example:trivial
    monodromy}.
    We present this example first because of its concreteness.
  Every page $F_t = \{\arg z_1 = t\} \cup K$ is diffeomorphic to the
  closed unit disk
  $\overline{\DD}^{2n-2} =
  \bigl\{(q_1+ip_1,\dotsc,q_{n-1}+ip_{n-1})\in \CC^{n-1}\bigm|\,
  \norm{\bfq+i\bfp} \le 1\bigr\}$
  which can be embedded into $\SS^{2n-1}$ using the inverse of the
  stereographic projection\footnote{The ``most obvious candidate'' for
    such an embedding, a map of the type
    $(x,y) \mapsto \bigl(x,y,\sqrt{1-(x^2+y^2)}\bigr)$ fails to be
    smooth along the boundary, and is thus not suitable for our
    purposes!}
  \begin{equation*}
    \iota_t\colon (\bfq+i\bfp) \mapsto
    \frac{1}{1 + \norm{\bfq + i\bfp}^2}\,
    \Bigl((1-\norm{\bfq + i\bfp}^2)\,e^{it};
    2\,(\bfq+i\bfp)\Bigr) \;.
  \end{equation*}
  The $1$-form $\beta := \alpha_0/\abs{g_1}$ is the binding form for
  the Liouville open book
  $\bigl(K, \vartheta, \restricted{d\beta}{F_t}\bigr)$.
  That this really defines a Liouville structure on the pages can be
  verified by pulling back $\beta$ with $\iota_t$ to $\DD^{2n-2}$.
  We obtain
  $\abs{z_1\circ \iota_t} = \frac{1-\norm{\bfq+i\bfp}^2}{1 +
    \norm{\bfq+i\bfp}^2}$
  and
  $\iota_t^* \alpha_0 = \frac{2}{(1 + \norm{\bfq+i\bfp}^2)^2}\,
  \lambda_0$ so that
  \begin{equation*}
    \iota_t^*\beta = \frac{\lambda_0}{1 - \norm{\bfq+i\bfp}^4} \;.
  \end{equation*}
  This is precisely the Liouville form on the unit disk given in
  Example~\ref{example: ideal Liouville domain}.(a).
  It follows that the page of the abstract open book is
  $(\overline{\DD}^{2n-2}, \omega)$ just as we wanted to show.
  Consider now the vector field
  \begin{equation*}
    Y = 2\pi\,\bigl(x_1\,\partial_{y_1} - y_1\,\partial_{x_1}\bigr)
  \end{equation*}
  on $\SS^{2n-1}$.
  Clearly, $Y$ vanishes along $K = \{x_1=y_1=0\}$, and it satisfies
  $d\vartheta(Y) = 2\pi$.
  Furthermore $\lie{Y}\beta = 0$, so that its flow preserves the
  Liouville structures induced by $d\beta$ on each page.
  We obtain that $Y$ is a spinning vector field and since its time-$1$
  flow is the identity on all of $\SS^{2n-1}$, it follows in
  particular that the monodromy of the Liouville open book is trivial.
  \textbf{(b)} The second open book decomposition~$(K, \vartheta)$ of
  the sphere is determined by the function
  $g_2(\z) = z_1^2+ \dotsm + z_n^2$.
  Remember that in Example~\ref{example: ideal Liouville domain}.(b)
  we described ideal Liouville domains on the closed unit disk
  cotangent bundles~$\overline{\DD}(T^*L)$ with Liouville form
  $\frac{1}{1-\norm{\bfp}^2}\,\lcan$.
  For the special case $L=\SS^{n-1}$ we can significantly simplify
  these manifolds by using the identification
  $\overline{\DD}(T^*\SS^{n-1}) \coloneqq \bigl\{(\bfq,\bfp)\in \RR^n
  \times \RR^n\bigm|\, \norm{\bfq} = 1, \, \bfq\perp\bfp, \,
  \norm{\bfp} \le 1\bigr\}$
  with Liouville form
  $- \frac{1}{1-\norm{\bfp}^2}\, \sum_{j=1}^n p_j\,dq_j$.
  We will now show that this ideal Liouville domain is the page of the
  abstract open book corresponding to $(K,\vartheta)$.
  We embed the unit disk bundle into $\SS^{2n-1}$ via
  \begin{equation*}
    \iota_t\colon (\bfq,\bfp) \mapsto
    \frac{\bigl(\bfq + i\bfp\bigr)\, e^{i t/2}}{\sqrt{1 + \norm{\bfp}^2}} \;.
  \end{equation*}
  The image of each such map is one of the pages.
  Pulling back $\alpha_0$, we obtain
  $\iota_t^* \alpha_0 = \frac{1}{2\,(1 + \norm{\bfp}^2)}\,
  \sum_{j=1}^{n-1} (q_j\,dp_j - p_j\,dq_j)$.
  This can be simplified using that the differential of
  $\langle \bfq, \bfp\rangle = 0$ is
  $\sum_{j=1}^{n-1} (q_j\,dp_j + p_j\,dq_j) = 0$ so that
  $\iota_t^* \alpha_0 = -\frac{1}{1 + \norm{\bfp}^2}\,
  \sum_{j=1}^{n-1} p_j\,dq_j$.
  We claim that $\beta := \alpha_0/\abs{g_2}$ is the binding form for
  a Liouville open book.
  Note that
  $\abs{g_2\circ\iota_t} = \frac{1 - \norm{\bfp}^2} {1 +
    \norm{\bfp}^2}$,
  so that
  $\iota_t^*\beta = - \frac{1}{1 - \norm{\bfp}^2} \, \sum_{j=1}^{n-1}
  p_j\,dq_j$, which is the Liouville form given on the domain above.
  This shows that the page of the abstract open book is indeed
  $\bigl(\overline{\DD}(T^*\SS^{n-1}), \omega\bigr)$.
  It remains to show that the monodromy is a generalized Dehn twist as
  we had already claimed in Example~\ref{example: contact open book
    decomposition for sphere}.
  Recall first that a \defin{Dehn twist} {on $T^*\SS^{n-1}$} can
  be written as follows:
  Identify the cotangent bundle of $\SS^{n-1}$ with the submanifold of
  $\RR^n\times \RR^n = \CC^n$ consisting of pairs of
  points~$(\bfq,\bfp)$ such that $\norm{\bfq} = 1$ and $\bfp
  \perp \bfq$.
  The canonical $1$-form~$\lcan$ on $T^*\SS^{n-1}$ is simply the
  restriction of the $1$-form $- \sum_{j=1}^n p_j\,dq_j$ to
  $T^*\SS^{n-1}$.
  Then we can write down the following type of symplectomorphisms with
  compact support
  \begin{equation*}
    \Phi\colon T^*\SS^{n-1} \to T^*\SS^{n-1},
    \begin{pmatrix}
      \bfq \\
      \bfp 
    \end{pmatrix}
    \mapsto
    \begin{pmatrix}
      \bfq\,\cos \rho + \frac{\bfp}{\norm{\bfp}}\, \sin \rho \\
      -\norm{\bfp}\,\bfq\,\sin \rho + \bfp\,\cos \rho
    \end{pmatrix} \;,
  \end{equation*}
  where $\rho(\bfq,\bfp) \coloneqq \rho(\norm{\bfp})$ is any smooth
  function that is $0$ for very large values of $\norm{\bfp}$ and is
  equal to $-\pi$ on a neighborhood of $0$.
  {Dividing} by $\norm{\bfp}$ in the definition of $\Phi$ is not
  problematic, because $\sin \rho$ vanishes close to the zero section
  of $T^*\SS^{n-1}$.
  A direct verification shows that $\Phi$ preserves the
  length~$\norm{\bfp}$, that it has compact support, and pulling back
  the canonical $1$-form using that $d\norm{\bfq}^2 = d1 = 0$ and that
  $\sum_j p_j \,dq_j = - \sum_j q_j\, dp_j$, we obtain
  \begin{equation*}
    \Phi^*\lcan = \lcan - \norm{\bfp}\, d\rho \;,
  \end{equation*}
  which implies that $\Phi$ is indeed a symplectomorphism, because
  $\rho$ only depends on $\norm{\bfp}$.
  Let us pull-back the generalized Dehn twist to the ideal Liouville
  domain $\bigl(\overline{\DD}(T^*\SS^{n-1}), \omega\bigr)$.
  As in Example~\ref{example: ideal Liouville domain}.(b), we stretch
  out the interior of the unit disk bundle to cover the full cotangent
  bundle using the map
  $(\bfq, \bfp) \mapsto \bigl(\bfq, \frac{1}{1-\norm{\bfp}^2}\,
  \bfp\bigr)$.
  The inverse of this map is $(\bfq, \bfp) \mapsto \bigl(\bfq,
  \frac{\sqrt{4\norm{\bfp}^2 +1} -1}{2\,\norm{\bfp}^2}\, \bfp\bigr)$.
  We pull back the Dehn twist to the ideal Liouville domain and find
  that it still is of the same form as on $T^*\SS^{n-1}$, only that
  the function $\rho(\norm{\bfp})$ needs to be replaced by
  $\rho\bigl(\frac{\norm{\bfp}}{1-\norm{\bfp}^2}\bigr)$.
  {Since the monodromy map of an abstract Liouville open book
    needs to be the identity only on the boundary of the domain, we
    weaken our definition by requiring that $\rho$ vanishes for
    $\norm{\bfp} = 1$, but not necessarily also on a neighborhood of
    $1$.
    Also instead of imposing that $\rho = -\pi$ on a small
    neighborhood of the zero section, it is sufficient for us that
    this condition is met \emph{on} the zero section itself (being
    careful to preserve the smoothness of $\Phi$).}
  We thank the referee for providing us with the following
    concise definition:

    \begin{definition}\label{definition: dehn twist}
      Choose any smooth function $g\colon [0,1] \to \RR$ such that
      $g(1) = \pi$.
      A \defin{Dehn twist on the ideal Liouville domain
        $\bigl(\overline{\DD}(T^*\SS^{n-1}), \omega\bigr)$} is a map
      of the form
      \begin{equation*}
        \Phi\colon \overline{\DD}(T^*\SS^{n-1}) \to \overline{\DD}(T^*\SS^{n-1}),
        \begin{pmatrix}
          \bfq \\
          \bfp
        \end{pmatrix}
        \mapsto
        \begin{pmatrix}
          \bfq\,\cos \rho + \frac{\bfp}{\norm{\bfp}}\, \sin \rho \\
          -\norm{\bfp}\,\bfq\,\sin \rho + \bfp\,\cos \rho
        \end{pmatrix} \;,
      \end{equation*}
      where $\rho(\bfq,\bfp)\coloneqq \rho(\norm{\bfp})$ can be
      written as $\rho(r) = r\, g(r^2) -\pi$.
    \end{definition}

    To show that this more general definition is suitable, we must verify that
  $\Phi$ is smooth along the $0$-section of $T^*\SS^{n-1}$.
  For this, observe that if $f\colon \RR\to \RR$ is a smooth
  even function, the composition $\bfp \mapsto f(\norm{\bfp})$
  will also be smooth.

  Expand  the trigonometric functions to see that
  \begin{equation*}
    \sin \rho(r) = - \sin \bigl(r\, g(r^2)\bigr)\quad \text{ and }\quad
    \cos \rho(r) = - \cos\bigl(r\, g(r^2)\bigr) \;.
  \end{equation*}
  The second function is clearly even. For the first one, notice that
  $-\sin \bigl(r\, g(r^2)\bigr)$ is odd and vanishes at $r=0$. We may
  therefore write it as $r\, h(r)$ for a smooth $h$ that necessarily
  needs to be even.
  This then implies that both $\frac{1}{r}\, \sin \rho(r) = h(r)$ and
  $ r\,\sin \rho(r) = r^2\, h(r)$ are well-defined and even, and
  thus $\Phi$ is everywhere smooth.
  {From this,} every Dehn twist lies in
  $\Diff_\p \bigl(\overline{\DD}(T^*\SS^{n-1}), \omega\bigr)$, and
  the space of Dehn twists is contractible. 

\vspace{0.3cm}

  Recall that the monodromy of the open book is obtained as the
  restriction to a page of the time-$1$ flow of a spinning vector
  field.
  A long but straight-forward computation shows that the vector
  field~$Y$ specified by Lemma~\ref{lemma: exists vector field to
    recover monodromy} is
  \begin{equation*}
    Y =  \pi\,\RealPart(g_2)\,\sum_{j=1}^n (y_j\, \frac{\p}{\p x_j} +
    x_j\, \frac{\p}{\p y_j})
    + \pi\,\ImaginaryPart(g_2)\,\sum_{j=1}^n (y_j\, \frac{\p}{\p y_j} -
    x_j\, \frac{\p}{\p x_j}) \;,
  \end{equation*}
  or equivalently using the Wirtinger formalism, we can write $Y$ as
  \begin{equation*}
    Y = \pi i\, g_2 \cdot \sum_{j=1}^n\bar z_j \, \frac{\p}{\p z_j}
    -  \pi i\, \bar g_2 \cdot \sum_{j=1}^n z_j \, \frac{\p}{\p \bar z_j} \;,
  \end{equation*}
  where we have used that $\frac{\p}{\p z_j} \coloneqq \frac{1}{2}\,
  \bigl(\frac{\p}{\p x_j} - i\, \frac{\p}{\p y_j}\bigr)$.
  The arguments explained in
  \cite{McDuffSalamonIntro}*{Exercise~6.20}, allow us to find the
  flow~$\Phi_t^Y$ of this vector field:
  Combining $\vartheta = g_2/\abs{g_2}$ with the normalization of $Y$ we
  obtain the equation
  \begin{equation*}
    \frac{g_2(\Phi^Y_t(\z))}{\abs{g_2(\Phi^Y_t(\z))}} = e^{2\pi it} \;.
  \end{equation*}
  We also see easily that the flow of $Y$ preserves $\abs{g_2}$,
  because $\lie{Y}\abs{{g_2}}^2 = 0$.
  Let $\z(t)$ be the trajectory of $Y$ starting at a point $\z(0)$.
  To simplify the notation write $g_0$ instead of
  $\abs{g_2\bigl(\z(0)\bigr)}$. Then $\z(t)$ is the solution to the
  ordinary differential equation
  \begin{equation*}
    \dot \z(t) = \pi i\, g_2\bigl(\z(t)\bigr)\cdot \bar \z(t)
    = \pi i\, e^{2\pi it}\,  g_0 \cdot \bar \z(t) \;.
  \end{equation*}
  Defining $\bfu(t) := e^{-\pi it}\, \z(t)$, we compute
  \begin{equation*}
    \dot \bfu(t) = -  \pi i\, \bfu(t) + e^{-\pi i t}\, \dot \z(t) = 
    - \pi i\, \bfu(t) +  \pi i\,  g_0 \cdot \bar \bfu(t)\;.
  \end{equation*}
  Splitting $\bfu$ into real and imaginary parts $\bfu_x + i \bfu_y$,
  the previous equation can be written as
  $\dot \bfu_x(t) = \pi\, (1 + g_0)\, \bfu_y(t)$ and
  $\dot \bfu_y(t) = - \pi\,(1-g_0)\, \bfu_x(t)$, which combines back
  to
  \begin{equation*}
    \ddot \bfu(t) = - \pi^2\,(1-g_0^2)\, \bfu(t) \;.
  \end{equation*}
  The general solution of this equation is
  $\bfu(t) = A_+\,e^{\pi ict} + A_-\,e^{-\pi ict}$ with
  $c := \sqrt{1-g_0^2}$ so that
  \begin{equation*}
    \z(t) = A_+\,e^{\pi i\, (c+1)\,t} +  A_-\,e^{-\pi i\, (c-1)\,t} \;,
  \end{equation*}
  where $A_+,A_-\in \CC^n$ are complex vectors.
  The coefficients are
  \begin{equation*}
    A_\pm = \frac{1}{2}\,
    \left(1\mp\sqrt{\frac{1-g_0}{1+g_0}}\right)\,\x(0) + \frac{i}{2}\,
    \left(1\mp\sqrt{\frac{1+g_0}{1-g_0}}\right)\,\y(0) \;.
  \end{equation*}
  so that we find for $\z(1)$
  \begin{equation*}
    \z(1) = - \z(0)\, \cos \pi\,\sqrt{1-g_0^2}
    + i\,\Bigl(\sqrt{\frac{1-g_0}{1+g_0}} \,\x(0)
    +  i\, \sqrt{\frac{1+g_0}{1-g_0}}\,\y(0)  \Bigr)\, \sin\pi\,
    \sqrt{1-g_0^2} \;.
  \end{equation*}
  To recover the monodromy of the abstract open book, restrict the
  time-$1$ flow of $Y$ to the $0$-page of the contact open book, and
  pull back the diffeomorphism obtained this way to the abstract page
  via the embedding~$\iota_0$.
  Recall first that
  $g_2(\iota_0(\bfq + i\bfp)) = \frac{1 - \norm{\bfp}^2} {1 +
    \norm{\bfp}^2}$.
  Then we get
  \begin{multline*}
    \iota_0^{-1}\circ \Phi_1^Y \circ \iota_0(\bfq + i\bfp) = -
    \Bigl(\bfq\, \cos \frac{2\pi\, \norm{\bfp}}{1 + \norm{\bfp}^2} +
    \frac{\bfp}{\norm{\bfp}}\, \sin \frac{2\pi\, \norm{\bfp}}{1 + \norm{\bfp}^2} \Bigr)\\
    - i\, \Bigl(- \norm{\bfp} \, \bfq\, \sin \frac{2\pi\,
      \norm{\bfp}}{1 + \norm{\bfp}^2} + \bfp\, \cos \frac{2\pi\,
      \norm{\bfp}}{1 + \norm{\bfp}^2}\Bigr) \;.
  \end{multline*}
  This is just a Dehn twist as in Definition~\ref{definition: dehn
    twist} with the function
  $\rho(\bfq, \bfp) := \frac{2\pi\, \norm{\bfp}}{1 + \norm{\bfp}^2} -
  \pi = \norm{\bfp}\, g\bigl(\norm{\bfp}^2\bigr) - \pi$,
  where $g(r)= 2\pi/(1 + r)$.
  Clearly $g(1) = 2\pi /2 = \pi$, as desired.
\end{example}

\begin{example}\label{example:trivial monodromy}
  We will now consider the case of a contact open book with trivial
  monodromy, and show that this manifold is symplectically filled by
  the stabilization of one of its pages.
  Using Theorem~\ref{theorem: equivalence abstract and embedded open
    books}, we will argue in the opposite direction, namely we take
  the stabilization of an arbitrary Liouville domain, and show that it
  contains a contact-type hypersurface that is supported by an open
  book with trivial monodromy and whose pages are isomorphic to the
  initial Liouville domain.
  We conclude, using that any contact manifold with an open book
  decomposition with trivial monodromy can be obtained via this
  construction.
  Let $(F, d\beta)$ be a (classical) Liouville domain with
  boundary~$\p F = K$ and with associated Liouville vector
  field~$X_L$, i.e.~so that $d\beta(X_L, \cdot) = \beta$.
  Choose a function $u \colon F \to [0, \infty)$ as in
  Lemma~\ref{lemma: function u for completion}, that is, $0$ is a
  regular value, $u^{-1}(0) = K$, and $du(X_L) < u$.
  In the case that $(F, d\beta)$ is a Weinstein domain with a Lyapunov
  function~$f$ for $X_L$ with $f^{-1}(C) = \p F$, we can simply set
  $u := C-f$.
  Let now $V \subset F \times \CC$ be the hypersurface defined by
  \begin{equation*}
    V = \{ (p, z) \in F \times \CC \, | \, u(p) - \abs{z}^2 = 0 \} \;.
  \end{equation*}
  From our hypothesis on $u$ it follows that $0$ is a regular value of
  $u(p)-\abs{z}^2$, so that $V$ is a closed embedded submanifold
  (touching the boundary of $F\times \CC$ from the inside).
  The manifold~$F \times \CC$ has an exact symplectic structure given
  by $d\bigl( \beta + \frac{1}{2}\, ( x\, dy - y\, dx) \bigr)$ where
  $z = x+iy$ denotes the coordinate on $\CC$.
  The corresponding Liouville field is
  $X_L + \frac{1}{2}\,\bigl(x\,\partial_x + y\, \partial_y\bigr)$.
  This vector field is transverse to $V$, because
  \begin{equation*}
    du(X_L) - \abs{z}^2 = du(X_L) - u < 0\;,
  \end{equation*}
  by the properties in Lemma~\ref{lemma: function u for completion},
  and it follows that the restriction of
  $\beta + \frac{1}{2}\, \bigl(x\, dy - y\, dx \bigr)$ to $V$ defines
  a contact form~$\alpha$ on $V$.

  \vspace{0.5cm}

  The open book we consider is obtained by taking the binding to be
  $K = K \times \{ (0,0) \} \subset F \times \CC$, and the defining
  map $f \colon V \to \CC$ by $f(p, z) = z$.
  It is not very difficult to check that $f$ really defines a smooth
  open book decomposition on $V$, so that we will only show that the
  pair $(\alpha, f)$ is a representation of a \emph{contact} open
  book.
  The closure of any page is diffeomorphic to $F$, since it is then
  the set of points
  $\{ (p, r\e^{i \vartheta} ) \in V \, | \, r \ge 0, p \in F \}$.
  This admits an ``obvious'' identification with $F$, given by
  \begin{equation*}
    p \mapsto \bigl(p, \sqrt{u(p)}\, \e^{i \vartheta}\bigr) \;,
  \end{equation*}
  but unfortunately this map fails to be smooth up to the boundary
  (compare to the footnote of Example~\ref{example: from open book to
    abstract open book}).
  Instead, we will need to pre-compose it with a
  homeomorphism~$\varphi \colon F\to F$ that is a diffeomorphism on
  the interior, and that maps the collar neighborhood
  $(-\epsilon, 0] \times K \to (-\epsilon, 0] \times K$ by
  $(s, x) \mapsto \bigl(g(s), x\bigr)$, where $g(s) = s^2$ for $s$
  near $0$, $g(s) = s$ for $s$ near $-\epsilon$ and $g'(s) > 0$ for
  $s < 0$ (this then extends as the identity of $F$ away from the
  collar neighborhood).
  We denote the composition~$\sqrt{u\circ \varphi}$ by $\tilde u$, and
  observe that
  \begin{equation*}
    \Phi\colon F\hookrightarrow V\times \CC,\,
    p \mapsto \bigl(p, \tilde u(p)\, \e^{i \vartheta}\bigr)
  \end{equation*}
  is a smooth embedding of $F$ into $V$.
  Equivalently, we could have treated this as a change of smooth
  structure at the boundary of $F$.
  (This is related to the discussion of smoothness immediately
  preceding Proposition~21 in \cite{GirouxIdealDomains}.)
  The resulting ideal Liouville form on the page is given by the
  restriction of $\alpha/\abs{z}$, which pulls-back to the $1$-form
  \begin{equation*}
    \lambda \coloneqq \Phi^*\Bigl(\frac{1}{\abs{z}}\, \alpha\Bigr)
    = \frac{1}{\tilde u}\, \beta
  \end{equation*}
  on $\Int F$.
  We now claim this gives $(F, d\lambda)$ an ideal Liouville
  structure.
  This requires that $v\,\lambda$ extends to a contact form on $K$ for
  any smooth function~$v\colon F\to [0,\infty)$ for which
  $K = v^{-1}(0)$ is a regular level set.
  The function~$\tilde u$ introduced above is such a function, and
  clearly $\tilde u\, \lambda$ agrees with the $1$-form~$\beta$ that
  is a contact form on $K$.
  To verify that $d\lambda$ is indeed symplectic on
  $F \setminus \p F$, write $r = \abs{z} = \sqrt{u}$ and compute
  \emph{in the interior of $F$}:
  \begin{equation*}
    r^{n+2} \left [ d \Bigl(\frac{1}{r}\beta \Bigr) \right ]^n
    =  r^2\, (d\beta)^n - \frac{n}{2}\, d(r^2) \wedge \beta \wedge (d\beta)^{n-1} 
    = u\, (d\beta)^{n} - \frac{n}{2}\, du \wedge \beta \wedge
    (d\beta)^{n-1} \;.
  \end{equation*}
  Now, contracting $X_L$ with $0 = du \wedge (d\beta)^n$, we obtain
  the identity
  \begin{equation*}
    0 = du(X_L)\, (d\beta)^n - n\, du \wedge \beta \wedge (d\beta)^{n-1} \;.
  \end{equation*}
  It now follows that
  \begin{equation*}
    r^{n+2} \left [ d \Bigl(\frac{1}{r}\beta \Bigr) \right ]^n 
    = \frac{1}{2}\,\Bigl(2u -  du(X_L)\Bigr)\, (d\beta)^n
  \end{equation*}
  is a positive volume form on $F$, because $u$ satisfies the
  assumptions in Lemma~\ref{lemma: function u for completion}, and we
  have that $u - du(X_L) > 0$.
  This shows that $(F, d\lambda)$ is symplectomorphic to the
  completion of $(F, d\beta)$.
  Finally, to compute the monodromy, we notice that by our
  construction, $\beta$ itself is a binding form on $V$, and thus
  $2\pi \partial_\vartheta$ is a spinning vector field.
  Its monodromy is indeed the identity map.
  Any Liouville page~$F$ can be used as the starting point for this
  construction.
  If, additionally, $F$ is a Weinstein domain, we obtain $V$ as the
  boundary of an explicit subcritical filling.
\end{example}

\bibliographystyle{amsalpha}
\bibliography{main}

\end{document}